\documentclass[12]{article}
\usepackage[english]{babel}
\usepackage{amsmath,enumerate}
\usepackage{amsthm,hyperref}
\usepackage{amssymb}
\usepackage{amscd}
\usepackage[latin1]{inputenc}
\usepackage[all]{xy}
\numberwithin{equation}{section}
\usepackage[usenames]{color}
\newtheorem{theorem}{Theorem}[section]
\newtheorem{definition}[theorem]{Definition}
\newtheorem{lemma}[theorem]{Lemma}
\newtheorem{proposition}[theorem]{Proposition}

\newtheorem{corollary}[theorem]{Corollary}
\newtheorem{remark}[theorem]{Remark}

\newcommand{\A}{\mathcal{A}}
\newcommand{\T}{\mathcal{T}}
\newcommand{\F}{\mathcal{F}}

\def\N{\mathbb N}
\def\op{\mathrm{op}}
\def\S{\mathcal{S}}
\def\Ab{\mathrm{Ab}}
\def\coker{\mathrm{Coker}}
\def\id{\mathrm{id}}
\def\lmod#1{#1\text{-}\mathrm{Mod}}

\def\mod#1{\mathrm{Mod}\text{-}#1}
\newcommand{\C}{\mathcal{C}}

\newcommand{\fg}{\mathrm{fg}}

\newcommand{\Gen}{\Gen}

\newcommand{\Hom}{\mathrm{Hom}}

\newcommand{\Ker}{\mathrm{Ker}}
\renewcommand{\Im}{\mathrm{Im}}

\newcommand{\Add}{\mathrm{Add}}
\newcommand{\add}{\mathrm{add}}
\newcommand{\tr}{\mathrm{tr}}

\def\M{\mathcal M}

\newcommand{\keywords}[1]{\noindent\textbf{Key words and phrases: } #1}
\newcommand{\MSC}[1]{\noindent\textbf{2010 Mathematics Subject Classification. } #1}

 \usepackage[left=4.2cm, right=4.2cm, top=4.2cm, bottom=4.2cm]{geometry}

\def\Q{\mathcal Q}

\def\Padd{\mathrm{P.Add}}
\def\I{\mathcal I}

\def\Ob{\mathrm{Ob}}
\def\Gen{\mathrm{Gen}}
\def\hom{\mathrm{Hom}}

\def\t{\mathbf{t}}

\def\X{\mathcal{X}}

\def\lim{\mathrm{lim}}
\def\proj{\mathrm{proj}}
\def\B{\mathcal{B}}
\def\Rej{\mathrm{Rej}}
\def\fpmod#1{\mathrm{mod}\text{-}#1}

\def\Y{\mathcal{Y}}
\def\E{\mathcal{E}}
\def\P{\mathcal{P}}

\title{Torsion pairs in categories of modules\\ over a preadditive category}

\author{Carlos Parra\footnote{The first named author was supported by CONICYT/FONDECYT/Iniciaci\'on/11160078} \and Manuel Saor\'{\i}n\footnote{The second named author was supported by the research projects from
the Ministerio de Econom\'{\i}a y Competitividad of Spain (MTM2016-77445-P) and the Fundaci\'on `S\'eneca' of
Murcia (19880/GERM/15), both with a part of FEDER funds.} \and Simone Virili\footnote{The third named author was supported by the Ministerio de Econom\'{\i}a y Competitividad of
Spain via a grant `Juan de la Cierva-formaci\'on'. He was also supported by the
Fundaci\'on `S\'eneca' of Murcia (19880/GERM/15) with a part of FEDER funds.}}

\begin{document}

\maketitle

\begin{abstract}
It is a result of Gabriel that hereditary torsion pairs in categories of modules are in bijection with certain filters of ideals of the base ring, called Gabriel filters or Gabriel topologies. A result of Jans shows that this bijection restricts to a correspondence between (Gabriel filters that are uniquely determined by) idempotent ideals and TTF triples. 
Over the years, these classical results have been extended in several different directions. In this paper we present a detailed and self-contained exposition of an extension of the above bijective correspondences to additive functor categories over small preadditive categories. In this context, we also show how to deduce parametrizations of hereditary torsion theories of finite type, Abelian recollements by functor categories, and centrally splitting TTFs.
\end{abstract}

\keywords Torsion pair, TTF triple, additive categories, Gabriel topology, Grothendieck topology, idempotent ideal, recollement. \\
\MSC  13D30, 16S90, 18E40, 18E05.

%\setcounter{tocdepth}{1}
%\tableofcontents

\section*{Introduction}
\addcontentsline{toc}{section}{Introduction}

{\bf Torsion theories} (also called {\bf torsion pairs}) were introduced by Dickson \cite{D} in the general setting of Abelian categories, taking as a model the classical theory of torsion Abelian groups. Given an Abelian category $\C$, a torsion pair $(\T,\F)$ in $\C$ is a pair of full subcategories satisfying the following axioms:
\begin{enumerate}[\rm ({Tors.}1)]
\item $\T={}^{\perp}\F$ (and $\F=\T^{\perp}$); where, for any class $\X$ of objects, we put \[
\X^{\perp}:=\{C \in \C:\C(X,C)=0, \text{ for all }X \in \X\}\quad\text{and}
\]
\[
{}^{\perp}\X:=\{C \in \C :\C(C,X)=0, \text{ for all }X \in \X\};
\] 
\item for each object $X$ of $\C$, there is an exact sequence
\begin{equation}\label{intro_standard_sequence}
\tag{\dag}\xymatrix{0 \ar[r] & T_X \ar[r] & X \ar[r] & F_X \ar[r] & 0,}
\end{equation}
with $T_X\in\T$ and $F_X\in\F$.
\end{enumerate}
Given a torsion pair ${\t}=(\T,\F)$, the class $\T$ (resp., $\F$) is said to be a {\bf torsion} (resp., {\bf torsionfree}) {\bf class}. Furthermore, ${\t}$ is said to be {\bf hereditary} if $\T$ is closed under taking subobjects (see, Sec.\ref{recall_torsion_subs}). (Hereditary) Torsion pairs have become a fundamental tool in the study of Grothendieck categories and their localizations; furthermore, they play an important role  in Algebraic Geometry and Representation Theory.

\medskip
Given a (unitary and associative) ring $R$, it is well-known since Gabriel's thesis \cite{Ga} that there is a one-to-one correspondence between {\bf Gabriel topologies} in $R$ (which are suitable filters of ideals of the ring) and hereditary torsion classes in $\mod R$ (see also \cite[Chapter VI]{S}). 
On the other hand, rings may be regarded as a special case of small {\bf preadditive categories} (i.e., small categories enriched over Abelian groups, see Sec.\ref{appendix_on_pre_additive}). Hence, the category of modules $\mod \A$ over a small preadditive category $\A$ naturally arises, see Sec.\ref{subs_modules}. There are many sources in the literature which deal with
this generalization (see, for example, \cite{mitchell}).
Given a preadditive category $\A$, the notion of ``linear Grothendieck topology'' on $\A$ (see Sec.\ref{grabriel_subs}), introduced in \cite{Lo,RG}, is an additive version of the notion of Grothendieck topology, which is of common use in Algebraic Geometry (see, for example, \cite{MM}). When applied to a preadditive category with just one object, one obtains the usual notion of Gabriel topology for rings. 

Our first general result uses these Grothendieck topologies to extend Gabriel's classical bijection:

\bigskip\noindent
{\bf Theorem A} (see Thm.\ref{thm.Gabriel bijection for small categories}){\bf .}
{\em 
Let $\A$ be a small preadditive category. Then there is an (explicit) one-to-one correspondence
between (linear) Grothendieck topologies on $\A$ and  hereditary torsion pairs in $\mod \A$.}

\bigskip
Let us remark that the above theorem could be deduced by the more general statement \cite[Prop.\,11.1.11]{Pr}, that applies to locally finitely generated Grothendieck categories; it appears without proof as Proposition 11.1.11 in \cite{Pr} (the proof follows the same lines of analogous results in \cite{G,Po}). Furthermore, Theorem A is in the same spirit of \cite[Prop.\,2.4 and 3.6]{AB}, where related characterizations of hereditary torsion classes are given in the more general setting of Grothendieck categories with a projective generator. For an implicit approach to the bijection of the theorem the reader is referred to \cite{Lo,L} (see also \cite[Sec.1.1]{RG}). 

\bigskip
Recall now that, in good enough Abelian categories (e.g., $\mod \A$ for a small preadditive category $\A$) a class $\T$ is a torsion class (resp., a torsionfree class) if and only if it is closed under taking quotients, extensions and coproducts (resp., subobjects, extensions and products). Hence, if we start with a hereditary torsion class $\T$ which is also closed under taking products, than $\T$ is both a torsion and a torsionfree class, for short, a {\bf TTF class}. In such case, a triple of the form $(^{\perp}\T,\T,\T^{\perp})$ is said to be a {\bf TTF triple} (see, Sec.\ref{recall_torsion_subs}); these objects have been introduced in categories of modules $\mod R$ over a ring $R$ by Jans \cite{Jans}, who showed that TTF triples are in bijection with idempotent ideals of $R$. Our second general result is to extend Jans' bijection from rings to small preadditive categories.

\bigskip\noindent
{\bf Theorem B} (see Thm.\ref{prop.direct proof}){\bf .}
{\em 
Let $\A$ be a small preadditive category. Then there is an (explicit) one-to-one correspondence
between idempotent ideals of $\A$ and TTF triples in $\mod \A$.}

\bigskip
After extending the Gabriel's and Jan's bijections to parametrize hereditary torsion pairs and TTF triples we concentrate on the following problem. {\bf Recollements} of Abelian categories are particularly nice decompositions of a given Abelian category by two other Abelian categories 
and, in nice enough situations (e.g., categories like $\mod \A$ for a preadditive category $\A$), they are known to be in bijection with TTF triples. In particular, given a TTF triple $(\C,\T,\F)$ in $\mod \A$, one can see $\mod \A$ as a recollement by the two Abelian categories $\T$ and $\C\cap \F$. We say that this is a {\bf recollement by categories of modules} if and only if both $\T$ and $\C\cap \F$ are equivalent to categories of module over some small preadditive categories. The following result extends some of the main results in \cite{PV}, characterizing those idempotent ideals of $\A$ that induce recollements by categories of modules:

\bigskip\noindent
{\bf Theorem C} (see Thm.\ref{main_thm_recollement}){\bf .}
{\em 
Let $\mathcal{A}$ be a small preadditive category. Then there are (explicit) one-to-one correspondences between:
\begin{enumerate}[\rm (1)]
\item equivalence classes of recollements of $\mod \A$ by categories of modules;
\item idempotent ideals of $\mathcal{A}$ that are  trace of sets of finitely generated projective modules;
\item the full subcategories of $\proj(\A)$ which are closed under coproducts and summands.
\end{enumerate}
Furthermore, up to replacing $\A$ by a Morita equivalent small preadditive category, the idempotent ideals in (2) are generated by a set of idempotent endomorphisms.
}

\bigskip
Finally, recall that a torsion pair $(\T,\F)$ is said to be {\bf split} if, for any object $X$, the canonical sequence \eqref{intro_standard_sequence} splits. Similarly, a TTF triple $(\C,\T,\F)$ is said to {\bf split} if both torsion pairs $(\C,\T)$ and $(\T,\F)$ split. Note that it might happen that only one of these torsion pair splits (see \cite{NS}), withouth the TTF triple being split. By a result of Jans \cite{Jans}, the bijection between TTF triples and idempotent ideals restricts to a second one between  central idempotents of a ring $R$ and splitting TTF triples in $\mod R$. As a last result, we extend this correspondence to small preadditive categories:

\bigskip\noindent
{\bf Corollary D} (see Sec.\ref{coro_central_id}){\bf .}
{\em 
Let $\A$ be a small preadditive category. Then there is an (explicit) one-to-one correspondence
between idempotents of the center $Z(\A)$ of $\mathcal{A}$ and split TTF triples in $\mod \A$.}

%\[
%\xymatrix{
%\cdots\ar[r]&0\ar[r]&0\ar[d]^{f}\ar[r]^{\alpha_0}&K\ar@{=}[d]^g\ar@{=}[r]^{\alpha_1}&K\ar[r]\ar[d]^h&0\ar[r]&\cdots\\
%\cdots\ar[r]&0\ar[r]&K\ar@{=}[r]^{\beta_0}&K\ar[r]^{\beta_1}&0\ar[r]&0\ar[r]&\cdots
%}
%\]

\section{Rings with several objects and their modules}

In this first section we recall some basic results and definitions about small preadditive categories (see Sec.\ref{appendix_on_pre_additive}). In particular, starting with a small preadditive category $\A$, we construct in a universal way a small additive and idempotent complete category $\widehat\A_{\oplus}$, called the Cauchy completion of $\A$.  In Sec.\ref{subs_modules} we introduce and study the category of modules $\mod \A$, showing in particular that two small preadditive categories have equivalent module categories if and only if they have equivalent Cauchy completions. 
In Sec.\,\ref{subs_traces_and_bimodules} we briefly recall the definition of trace of a class of modules on a given module. Furthermore, after recalling some basic properties of bimodules, we show how to define the trace on bimodules. In Sec.\ref{subs_locally_coh}, we obtain a characterization of when a category of modules over a small preadditive category $\A$ is locally coherent. We conclude by recalling in Sec.\ref{subs_centre} the notion of centre of a preadditive category.

\subsection{Preadditive categories and Cauchy completion} \label{appendix_on_pre_additive}

We denote by $\Padd$ (resp., $\Add$) the ($2$-)category of {\bf preadditive} (resp., {\bf additive}) {\bf categories}. For the rest of this subsection, $\A$ will denote a {\bf small} preadditive category. 

%{\bf Hemos usado dos cosas  diferentes:  la `additive closure' y la `idempotent completion', esta \'ultima sin definir. Esta \'ultima se est\'a suponiendo aditiva, por lo que yo propongo que le cambiemos el nombre a `additive idempotent completion' (de una categor\'ia pre-additiva), que ser\'a por definici\'on  la `idempotent completion' de la `additive closure' de $\A$. El resultado sobre categor\'ias de m\'odulos debe decir que las tres son equivalentes. Ruego Simone que modifiques esto en ese sentido. No tengo ahora en mente ninguna referencia concreta, pero es muy conocido que la categor\'ia de m\'odulos sobre una categor\'ia aditiva y sobre la de su `idempotent completion' son equivalentes via la restricci\'on de escalares.}

\smallskip
We define the {\bf additive closure} $\widehat \A$ of $\A$ as follows:
\begin{enumerate}[--]
\item $\Ob(\widehat\A):=\{(a_1,\dots,a_n): n\in\N,\, a_i\in \Ob(\A)\}$;
\item given $n$, $m\in\N$ and $a=(a_1,\dots,a_n)$, $b=(b_1,\dots,b_m)\in\Ob(\widehat{\A})$, 
\[
{\widehat{\A}}(a,b):=\{(r_{i,j}):r_{i,j}\colon a_i\to b_j, \text{ with } i=1,\dots,n,\, j=1,\dots,m\}.
\]
\item composition is given by the usual row-by-column multiplication of matrices.
\end{enumerate}
It is the well-known that $\widehat{\A}$ is a small additive category, where the coproduct of two objects $(a_1,\dots,a_n)$ and $(b_1,\dots,b_m)$ is given by $(a_1,\dots,a_n,b_1,\dots,b_m)$. Furthermore, the inclusion $\iota\colon \A\to \widehat{\A}$ such that $a\mapsto (a)$ is universal in a suitable sense (in particular $\A\cong \widehat{\A}$ if $\A$ was already additive):

\begin{lemma}\label{lemma_preadditive_additive closure}
Let $\A$ be a small preadditive category and $\B$ an additive category (not necessarily small), then $\iota\colon \A\to \widehat{\A}$ induces an equivalence of categories (here we are using the $2$-categorical structure of $\Padd$)
\[
-\circ\iota\colon \Add(\widehat{\A}, \B)\to \Padd(\A,\B).
\]
%which is surjective on objects. In other words, given an additive functor $f\colon \A\to \B$ there exists a functor $g\colon \widehat{\A}\to \B$ such that $g\circ \iota=f$, and such $g$ is unique up to a unique natural isomorphism.
\end{lemma}
\begin{proof}
Let $f\colon \A\to \B$ be an additive functor and define 
\[
g\colon \widehat\A\to \B\quad \text{as}\quad g(a_1,\dots,a_n):=f(a_1)\sqcup\cdots\sqcup f(a_n).
\] 
Clearly, $g\circ\iota=f$. Consider now a second functor $g'\colon \widehat{\A}\to \B$ and let us sketch an argument to show that there is an isomorphism
\begin{align}
{\Add(\widehat{\A}, \B)}(g,g')&\to {\Padd(\A,\B)}(f ,g'\circ \iota)\label{map_to_ff_pre_to_add}\\
\notag \alpha&\mapsto \alpha *\iota,
\end{align}
where $*$ denotes the horizontal composition of natural transformations. 
%Indeed, consider $\alpha\colon g\to g'\colon\widehat{\A}\to \B$. Any given $a=(a_1,\dots,a_n)\in \Ob(\widehat{\A})$ is the coproduct in $\widehat{\A}$ of its ``components", that is, $(a_1,\dots,a_n)= (a_1)\sqcup\ldots\sqcup(a_n)$ so, by additivity, $g(a)= g(a_1)\sqcup\ldots \sqcup g(a_n)$ and $g'(a)= g'(a_1)\sqcup\ldots \sqcup g'(a_n)$. Hence, $\alpha_a\colon g(a)\to g'(a)$ can be represented as an $n\times n$ matrix $(\alpha_a^{i,j})$, where $\alpha_a^{i,j}\colon g(a_j)\to g'(a_i)$. By the naturality of $\alpha$ it follows that this matrix is necessarily diagonal, that is $\alpha_a^{i,j}\neq 0$ implies $i=j$. For this, given $1\leq i\leq n$, consider the inclusion $(a_i)\to (a_1,\dots,a_n)$, then there is a commutative diagram
%\begin{equation}\label{alpha_is_diagonal}
%\xymatrix{
%g(a_i)\ar[r]^{\alpha_a^{i,i}}\ar[d]&g'(a_i)\ar[d]\\
%g(a_1,\dots,a_n)\ar[r]^{\alpha_a}&g'(a_1,\dots,a_n)
%}
%\end{equation}
%showing that $\alpha_a^{i,j}=0$ for all $j\neq i$. We can now verify that the map in \eqref{map_to_ff_pre_to_add} is an isomorphism. 
Indeed, let $\alpha\colon g\to g'$ be such that $\alpha*\iota=0$. For any $a\in \widehat{\A}$, $\alpha_a\colon g(a)\to g'(a)$ can be represented by a diagonal matrix and the diagonal entries of this matrix are trivial since $\alpha_a^{i,i}$ is conjugated to $(\alpha*\iota)_{a_i}=0$. 

On the other hand, given $\beta \colon g\circ \iota \to g'\circ \iota$, define $\alpha\colon g\to g'$ as follows: for $a=(a_1,\dots,a_n)\in\widehat{\A}$, $\alpha_a\colon g(a)\to g'(a)$ is the diagonal $n\times n$ matrix $(\alpha_a^{i,j})$ such that $\alpha_a^{i,i}:=\beta_{a_i}\colon g\iota(a_i)\to g'\iota(a_i)$.  It is easy to see that  $\alpha$ is a natural transformation and, clearly,  $\alpha*\iota=\beta$.
\end{proof}

By the above lemma, the $2$-category of small additive categories is reflective in the $2$-category of small preadditive categories, that is, there is a ($2$-)functorial way to make a small preadditive category into an additive category. A second application will be given in the next subsection: we will apply the lemma with $\B=\Ab$ to show that $\A$ is Morita equivalent to $\widehat \A$.

\smallskip
Suppose now $\A$ is an additive category; we define the {\bf idempotent completion} $\A_{\oplus}$ of $\A$ as follows:
\begin{enumerate}[--]
\item $\Ob(\A_{\oplus}):=\{(a,r): a\in \Ob(\A),\, r\colon a\to a \text{ such that }r^2=r\}$;
\item given $(a,r),\, (b,s)\in\Ob (\A_\oplus)$, 
\[
{\A_\oplus}((a,r),(b,s)):=\{t\colon a\to b: t=str\}.
\]
\item composition is as expected. Note that the identity of $(a,r)$ is $r\colon a\to a$. 
\end{enumerate}
It is well-known, and easy to verify, that $\A_\oplus$ is a small additive category where idempotents split due to the following fact: given $a\in \Ob (\A)$ and an idempotent $r\colon a\to a$, there is the following decomposition in $\A_{\oplus}$
\[
(a,\id_a)=(a,r)\oplus (a,\id_a-r)
\]
where clearly $(\id_a-r)=(\id_a-r)^2$ is an idempotent and the inclusions in the coproduct are given by $r\colon (a,r)\to (a,\id_a)$ and $(\id_a-r)\colon (a,\id_a-r)\to (a,\id_a)$. Furthermore, the inclusion $\iota\colon \A\to \A_{\oplus}$ such that $a\mapsto (a,\id_a)$ is universal in a suitable sense (in particular, $\A\cong \A_{\oplus}$ if idempotents split in $\A$):

\begin{lemma}\label{lemma_additive_idempotent closure}
Let $\A$ be a small additive category and $\B$ an  idempotent complete additive category, then $\iota\colon \A\to \A_{\oplus}$ induces an equivalence of categories (here we are using the $2$-categorical structure of $\Add$)
\[
-\circ\iota\colon \Add(\A_{\oplus}, \B)\to \Add(\A,\B).
\]
%which is surjective on objects. In other words, given an additive functor $f\colon R\to S$ there exists a functor $g\colon R_{\oplus}\to S$ such that $g\circ \iota=f$, and such $g$ is unique up to a unique natural isomorphism.
\end{lemma}
\begin{proof}
Let $f\colon \A\to \B$ be an additive functor and define a functor $g\colon \A_{\oplus}\to \B$ as follows: given $(a,r)\in \Ob(\A_{\oplus})$, where $r$ is an idempotent of $a\in \A$, then $f(r)$ is an idempotent of $f(a)\in \B$, hence $f(a)=b_1\oplus b_2$ and $f(r)=\pi_1\epsilon_1$, where $\pi_1$ and $\epsilon_1$ are the projection and inclusion relative to $b_1$. We then let $g(a,r):=b_1$. Clearly, $g\circ \iota=f$. Consider a second functor $g'\colon \A_\oplus\to \B$ and let us verify that there is an isomorphism
\begin{align}
{\Add(\A_\oplus, \B)}(f,f')&\to {\Add(\A,\B)}(f\circ \iota ,f'\circ \iota)\label{map_to_ff_add_to_id}\\
\notag \alpha&\mapsto \alpha *\iota.
\end{align}
Indeed, consider $\alpha\colon f\to f'\colon \A_\oplus\to \B$ such that $\alpha*\iota=0$. For any $(a,r)\in \Ob(\A_{\oplus})$ there is a commutative diagram in $\A_{\oplus}$
\[
\xymatrix@C=80pt{
(a,r)\ar[r]^{\epsilon_{(a,r)}:=r}\ar[d]_{\id_{(a,r)}:=r}&(a,\id_a)\ar@/_-12pt/[dl]^-{\ \ \ \pi_{(a,r)}:=r}\\%\ar[d]^{\id_{(a,\id_a)}=\id_a}\\
(a,r)&%(a,\id_a)
}
\]
(here $\pi_{(a,r)}$ and $\epsilon_{(a,r)}$ are represented by the same idempotent morphism $r\colon a\to a$ but the former is a morphism $(a,\id_a)\to (a,r)$, while the latter is a morphism $(a,r)\to (a,\id_a)$). Now, $\alpha_{(a,\id_a)}=(\alpha*\iota)_a=0$ and so $f'(\epsilon_{(a,r)})\alpha_{(a,r)}=\alpha_{(a,\id_a)} f(\epsilon_{(a,r)})=0$, so that $\alpha_{(a,r)}=f'(\pi_{(a,r)})f'(\epsilon_{(a,r)})\alpha_{(a,r)}=0$. Finally, let $\beta\colon f\circ \iota \to f'\circ \iota$ and define $\alpha\colon f\to f'$ as follows: given $(a,r)\in \Ob (\A_{\oplus})$, let $\alpha_{(a,r)}:=f'(\pi_{(a,r)})\beta_a f(\epsilon_{(a,r)})$. Clearly, $\alpha*\iota=\beta$.
\end{proof}
%
%Applying the above proposition to any small idempotent complete additive category $S$, one obtains that the $2$-category of small idempotent complete additive categories is a reflective full $2$-subcategory of $\Add$. On the other hand, applying the same proposition to $S=\Ab$ one gets the following useful consequence:
%
%\begin{corollary}\label{IC_morita}
%Given a small additive category $R$, the change of base induced by the inclusion $\iota\colon R\to R_{\oplus}$
%\[
%\iota_!: (R^\op,\Ab)\rightleftarrows (R_{\oplus}^\op,\Ab) :\iota^*
%\]
%gives a Morita equivalence $(R^\op,\Ab)\cong (R_\oplus^\op,\Ab)$. 
%\end{corollary}
%
%Finally, for a given small preadditive category $R$, define its {\bf Cauchy completion} 
%\[
%\widehat R:=\add(R)_{\oplus},
%\] 
%the idempotent completion of the additivization. Putting together the Corollaries \ref{ADD_morita} and \ref{IC_morita} we obtain:
%
%\begin{corollary}\label{coro_eq_cauchy}
%Given a small additive category $R$, the change of base induced by the inclusion $\iota\colon R\to \widehat R$
%\[
%\iota_!: (R^\op,\Ab)\rightleftarrows (\widehat R^\op,\Ab) :\iota^*
%\]
%gives a Morita equivalence $(R^\op,\Ab)\cong (\widehat R^\op,\Ab)$. 
%\end{corollary}

By the above lemma, the $2$-category of idempotent complete small additive categories is reflective in the $2$-category of small additive categories, that is, there is a ($2$-)functorial way to make a small additive category idempotent complete. A second application will be given in the next subsection: we will apply the lemma with $\B=\Ab$ to show that $\A$ is Morita equivalent to $\A_\oplus$. 

\medskip
Given a preadditive category $\A$, the idempotent completion of the additive closure $\widehat \A_\oplus$ of $\A$ is usually referred to as the {\bf Cauchy completion} of $\A$ (see \cite{Lawvere}). Furthermore, $\A$ is said to be {\bf Cauchy complete} if it is equivalent to its Cauchy completion.

\subsection{Modules and Morita equivalence}\label{subs_modules}

A {\bf right} (resp., {\bf left}) {\bf module} $M$ over a small preadditive category $\A$ is an (always additive) functor $M\colon \A^{\op}\to \Ab$ (resp., $M\colon \A\to \Ab$). A morphism (a natural transformation) $\phi\colon M\to N$ between right $\A$-modules consists of a family of morphisms
$\phi_a\colon M(a)\to N(a)$ (of Abelian groups), with $a$ ranging in $\Ob (\A)$, such that the following squares commute for all $(r\colon a\to b)\in \A$:
\[
\xymatrix{
M(a)\ar[r]^{\phi_a}&N(a)\\
M(b)\ar[u]^{M(r)}\ar[r]^{\phi_b}&N(b).\ar[u]_{N(r)}
}
\]
We denote by  $\mod \A$ (resp., $\lmod \A$) the category  of right (resp., left) $\A$-modules. Given two right (resp., left) $\A$-modules $M$ and $N$, we denote by $\hom_\A(M,N)$ their group of morphisms in $\mod \A$ (resp., $\lmod \A$). As a natural example of right module over $\A$ one can consider the representable modules
\begin{equation}\label{representables_definition}
H_a:=\A(-,a)\colon \A^{\op}\to \Ab;
\end{equation}
for any $a\in\Ob(\A)$.

\begin{lemma}\label{mod_is_groth}
Let $\A$ be a small preadditive category. Then $\mod \A$ is a Grothendieck category with a family of small projective generators. 
\end{lemma}

The above lemma is well-known but nevertheless let us give a sketchy proof. The zero object $0$ in $\mod \A$ is the constant functor $a\mapsto 0$, for all $a\in\Ob(\A)$ and co/kernels are constructed componentwise. 
%Given a morphism $\phi\colon M\to N$ in $\mod \A$, the (co)kernel of $\phi$ is constructed sending $a\in\Ob(\A)$ to the (co)kernel of $\phi_a$ and sending $(r\colon a\to b)\in \A$ to the universal map $\Ker(\phi_a)\to \Ker(\phi_b)$ (resp., $\coker(\phi_a)\to \coker(\phi_b)$). 
As a consequence, a sequence $0\to N\to M\to M/N\to 0$ in $\mod \A$ is short exact if and only if $0\to N(a)\to M(a)\to (M/N(a))(=M(a)/N(a))\to 0$ is a short exact sequence in $\Ab$, for all $a\in \Ob(\A)$.  Furthermore, given a morphism $\phi$ in $\mod \A$, the canonical map $\varphi\colon\coker(\Ker(\phi))\to \Ker(\coker(\phi))$ is an isomorphism since, for any $a\in \Ob(\A)$, the map $\varphi_a$ is an isomorphism in $\Ab$.

Furthermore, arbitrary co/limits are induced componentwise by those in $\Ab$.  Hence, $\mod \A$ is a
bicomplete Abelian category where products and direct limits are  exact.

To see that $\mod \A$ is Grothendieck, it remains to describe a family of generators. In fact, one can find a family of finitely generated (=finitely presented) projective generators. To give a complete description of such modules, let us introduce the following notation: for a morphism $\alpha \colon x\to  y$ in $\mathcal{A}$ we let 
\[
\alpha\mathcal{A}:= \Im(\alpha\circ-\colon H_x\to  H_y)\leq H_y.
\] 
%That is, $\alpha\mathcal{A}$ is a subfunctor of $H_y:=\A(-,a)\colon \A^{\op}\to \Ab$ such that, given $a\in\Ob(\A)$,
%\[
%(\alpha\mathcal{A})(a)=\{\beta\in\mathcal{A}(a,y):\beta =\alpha\circ\gamma\text{, for some }\gamma\in\mathcal{A}(a,x)\}.
%\]
In particular, $\id_x\A=H_x$, for all $x\in \Ob(\A)$.
%The proof is omitted as it is an  application of (the additive version of) the Yoneda Lemma and some standard computations. 

\begin{lemma}\label{description_fpp_lemma}
Let $\mathcal{A}$ be a small preadditive category, let $P$ be a right $\mathcal{A}$-module, and consider the following assertions:
\begin{enumerate}[\rm (1)]
\item $P$ is finitely generated projective;
\item $P$ is isomorphic to $\epsilon\mathcal{A}$, for some idempotent endomorphism $\epsilon\in\mathcal{A}(x,x)$;
\item $P$ is isomorphic to $H_x$, for some $x\in\Ob(\A)$.
\end{enumerate}
The implications $(3)\Rightarrow(2)\Rightarrow(1)$ always hold true. On the other hand, $(1)\Rightarrow (2)$ holds if $\A$ is additive, while $(2)\Rightarrow (3)$ holds if $\A$ is idempotent complete. In particular, all the assertions are equivalent if $\A$ is Cauchy complete.\end{lemma}
\begin{proof}
(3)$\Rightarrow$(2) is clear since $H_x\cong \id_x\A$, and $\id_x$ is an idempotent in $\A(x,x)$.

\noindent\smallskip
(2)$\Rightarrow$(1). The fact that $H_x$ is finitely generated and projective is a consequence of (the additive version of) the Yoneda Lemma and some standard computations. Hence,  it is enough to prove that the inclusion $\iota\colon \epsilon\mathcal{A}\hookrightarrow H_x$ is a section in $\mod \A$. Indeed, for each $a\in\Ob(\A)$, we define $\pi_a\colon H_x(a)=\mathcal{A}(a,x)\to (\epsilon\mathcal{A})(a)$ by  $\pi_a(\beta)=\epsilon\circ\beta$. It is routine to check that the $\pi_a$'s define a natural transformation $\pi\colon H_x\to \epsilon\mathcal{A}$ that is a retraction for the inclusion. 

\smallskip\noindent
(1)$\Rightarrow$(2), when $\A$ is additive. Being $P$ finitely generated, we have a retraction 
\[
\xymatrix{
\rho\colon \coprod_{i=1}^nH_{x_i}\twoheadrightarrow P.
}
\] 
By additivity, the coproduct $x=\coprod_{i=1}^nx_i$ exists in $\mathcal{A}$, so we have an isomorphism  $\coprod_{i=1}^nH_{x_i}\cong H_x$, and we identify $\rho$ as a retraction $\rho\colon H_x\twoheadrightarrow P$. By choosing a section $\lambda \colon P\to  H_x$ for $\rho$, we get an idempotent endomorphism $\epsilon:=\lambda\circ\rho$ of $H_x$ whose image is isomorphic to $P$. By the Yoneda Lemma, we then get $\epsilon=\epsilon^2\in\hom_\A(H_x,H_x)\cong (H(x))(x)=\mathcal{A}(x,x)$, so that $P\cong\epsilon\mathcal{A}$. 

\smallskip\noindent
(2)$\Rightarrow$(3), when $\A$ is idempotent complete. As $\epsilon$ is an idempotent, and $\A$ is idempotent complete, there are $y\in\Ob(\A)$, $\rho\colon x\to y$ and $\lambda\colon y\to x$ such that $\epsilon=\lambda\circ\rho$ and $\rho\circ\lambda=\id_y$. Hence, $P\cong \epsilon\A\cong \rho\A=H_y$.
\end{proof}

As a consequence of the above lemma one can give a second description of the Cauchy completion $\widehat \A_\oplus$ of a preadditive category $\A$. For that, let us introduce the following notation: for a class $\S$ of objects  we let
\[
\add(\S):=\{\text{summands of finite coproducts of objects in $\S$}\}.
\]

\begin{corollary}\label{application_yoneda}
Let $\A$ be a small preadditive category and denote by $\proj(\A)$ the class of finitely generated projective $\A$-modules. Then, 
\begin{enumerate}[\rm (1)]
\item $\A$ is equivalent to the full subcategory of $\proj(\A)$ spanned by the representables;
\item $\widehat \A$ is equivalent to the full subcategory of $\proj(\A)$ spanned by the finite coproducts of representables;
\item $\widehat\A_{\oplus}$ is equivalent to $\proj(\A)=\add(H_a:a\in\Ob(\A))$. 
\end{enumerate}
\end{corollary}
\begin{proof}
(1) is a consequence of the Yoneda Lemma. In fact, the functor 
\[
y\colon \A\to \mod \A,
\] 
that maps an object $a\in\Ob(\A)$ to $H_a$, is fully faithful and we have already observed that each of the $H_a$'s is finitely generated and projective. 

\smallskip\noindent
(2) Consider a functor 
\[
\widehat y\colon \widehat \A\to \mod \A
\] 
that maps an object $(a_1,\dots,a_k)\in \Ob(\widehat \A)$ to the coproduct $\coprod_{i=1}^kH_{a_k}$, and that sends a morphism in $\Ob(\widehat \A)$ (which is by definition a suitable matrix) to the morphism between coproducts represented by the corresponding matrix. This functor is clearly fully faithful, so that (2) also follows. 

\smallskip\noindent
(3) Let us introduce some notation first: given $(a,r)\in\Ob(\widehat\A_{\oplus})$ (that is, $a\in \Ob(\widehat\A)$ and $r$ an idempotent in $\widehat\A(a,a)$) the morphism $\widehat y(r)\colon \widehat y(a)\to \widehat y(a)$ is an idempotent endomorphism in $\proj(\A)$, so we obtain the following epi-mono factorization in $\proj(\A)$:
\[
\xymatrix@C=50pt@R=3pt{
 \widehat y(a)\ar[rr]^{\widehat y(r)}\ar@/_5pt/[dr]_{\pi_r}&& \widehat y(a)\\
&P_r\ar@/_5pt/[ur]_{\iota_r}&
} 
\]
where $\pi_r\iota_r=\id_{P_r}$, and where $P_r$ is finitely generated projective, as it is a summand of the finitely generated projective object $\widehat y(a)$. We can now define the following functor:
\[
\widehat y_{\oplus}\colon \widehat \A_{\oplus}\to \mod \A
\] 
mapping an object $(a,r)$ in $\widehat \A_{\oplus}$ to $P_r$, and a morphism $f\colon (a,r)\to (b,s)$ to $\pi_s\circ\widehat y(f)\circ\iota_r$. Let us verify that $\widehat y_{\oplus}$ is an equivalence. Indeed, any $P\in \proj(\A)$ is a summand of a finite coproduct of representables and so, using the equivalence proved in part (2), there is an object $a$ in $\widehat \A$ and an idempotent endomorphism $r\in \widehat\A(a,a)$ such that $P= \Im(\widehat y(r))$, so that $P=P_r$. To conclude, let $(a,r)$ and $(b,s)$ be objects in $\widehat\A_\oplus$ and let us verify that the following homomorphism is bijective:
\begin{align*}
\widehat\A_\oplus((a,r),(b,s))&\to \hom_\A(P_r,P_s)\\
f&\mapsto \pi_s\circ\widehat y(f)\circ\iota_r.
\end{align*}
In fact, given $f\in \widehat\A_\oplus((a,r),(b,s))$ (so $f=s\circ f\circ r$) such that $ \pi_s\circ\widehat y(f)\circ\iota_r=0$, then $\widehat y(f)= \widehat y(s)\circ \widehat y(f)\circ \widehat y(r)= \iota_s\circ\pi_s\circ \widehat y(f)\circ \iota_r\circ\pi_r=0$, so $f=0$ by (2). On the other hand, given $\phi \in  \hom_\A(P_r,P_s)$, let $\tilde\phi:=\iota_s\circ \phi\circ \pi_r\colon \widehat y(a)\to \widehat y(b)$ and let $f\colon a\to b$ be such that $\widehat y(f)=\tilde\phi$. Then $f=s\circ f\circ r$, so $f$ can be viewed as a morphism $(a,r)\to(b,s)$ in $\widehat\A_\oplus$ and, as such, $\widehat y_\oplus(f)=\pi_s\circ\widehat y(f)\circ\iota_r=   \pi_s\circ \iota_s\circ \phi\circ \pi_r\circ\iota_r =\phi$.
%, to $\widehat y_{\oplus}(a,r):=\Im(\widehat y(r)\colon \widehat y(a)\to \widehat y(a))$. A morphism $t\colon (a,r)\to (b,s)$ in $\widehat\A_\oplus$ is mapped to $\widehat y(st)\colon \widehat y_{\oplus}(a,r)\to \widehat y_{\oplus}(b,s)$. Again, it is an exercise to verify that this functor is fully faithful. 
\end{proof}

Consider now an additive functor $\phi\colon \A \to \B$ between two small preadditive categories. Then, $\phi$ induces a {\bf restriction of scalars functor}
\[
\phi_*\colon \mod{\B}\to \mod \A\qquad\text{such that}\qquad M\mapsto M\circ\phi.
\]
It is easy to verify that $\phi_*$ is exact and that it commutes with co/products so, by the Special Adjoint Functor Theorem (see, for example, \cite[Thm. 3.3.4]{Bo}), it has a left adjoint, called the {\bf extension of scalars}, and a right adjoint, called the {\bf coextension of scalars}, denoted respectively by $\phi^*$ and $\phi^!$. 
\[
\xymatrix{
\mod\B\ar[rr]|{\phi_*}&&\mod \A\ar@/_15pt/[ll]_{\phi^!}\ar@/_-15pt/[ll]^{\phi^*}
}
\]
As a corollary of Lemm.\,\ref{lemma_preadditive_additive closure} and \ref{lemma_additive_idempotent closure} we can give a precise relation between $\mod\A$, $\mod{\widehat\A}$, and $\mod {\widehat \A_\oplus}$:

\begin{corollary}\label{ADD_morita}
Given a small preadditive category $\A$, consider the inclusions $\iota\colon \A\to \widehat{\A}$ and $\iota'\colon \widehat\A\to \widehat{\A}_\oplus$. The restrictions of scalars along $\iota$ and $\iota'$ are both equivalences. As a consequence, two small preadditive categories $\A$ and $\B$ are Morita equivalent if and only if there is an equivalence of categories $\widehat\A_\oplus\cong\proj(\A)\cong \proj(\B)\cong \widehat\B_\oplus$.
\end{corollary}

Let us conclude this subsection with the following remark, where an object $a\in\Ob(\A)$ is a {\bf $\oplus$-generator} if $\A=\add(a)$:

\begin{remark}\label{rem_morita_cauchy}
By the above corollary, we obtain the following bijections:
{\small 
\[
\xymatrix@R=12pt@C=-60pt{
{\left\{\begin{matrix}\text{(Ab.$3$) Abelian categories with a set}\\ \text{of small projective generators}\\ \text{up to equivalence}\end{matrix}\right\}}\ar@{<->}[rr]^(.55){1:1}\ar@{<->}[rd]^(.65){1:1}&&{\left\{\begin{matrix}\text{Small Cauchy complete}\\ \text{additive categories}\\ \text{up to equivalence}\end{matrix}\right\}}\\
&{\left\{\begin{matrix}\text{Preadditive categories up to Morita equivalence}\end{matrix}\right\}}\ar@{<->}[ru]^(.35){1:1}.
}
\]
}
Furthermore, these bijections induce, by restriction, the following ones:
{\small 
\[
\xymatrix@R=15pt@C=-40pt{
{\left\{\begin{matrix}\text{(Ab.$3$) Abelian categories with a}\\ \text{small projective generator}\\ \text{up to equivalence}\end{matrix}\right\}}\ar@{<->}[rr]^(.51){1:1}\ar@{<->}[rd]^(.65){1:1}&&{\left\{\begin{matrix}\text{Small Cauchy complete additive}\\ \text{categories with a $\oplus$-generator}\\ \text{up to equivalence}\end{matrix}\right\}}\\
&{\left\{\begin{matrix}\text{Rings up to Morita equivalence}\end{matrix}\right\}}\ar@{<->}[ru]^(.35){1:1}
}
\]
}
\end{remark}

\subsection{Traces and bimodules}\label{subs_traces_and_bimodules}

Given a class $\S$ of $\A$-modules and an $\A$-module $M$, we can construct a submodule $\tr_\S(M)$ of $M$ such that any map $S\to M$, with $S\in \S$, factors through the inclusion $\tr_\S(M)\to M$:

\begin{definition}
Let $\mathcal{S}$ be a class of right $\mathcal{A}$-modules and $M$ a right $\mathcal{A}$-module, then the sum of the submodules of $M$ of the form $\Im(f)$, for some morphism $f\colon S\to M$ in $\mod \A$, with $S\in\mathcal{S}$, is called the {\bf trace of $\mathcal{S}$ in $M$} and denoted by $\tr_\mathcal{S}(M)$.
%, that is,
% \[
% \tr_{\S}(M):=\sum_{S\in \S}\ \  \sum_{f\colon S\to M}\Im(f). 
% \]
 \end{definition}
 
In the following lemma we see that the assignment $M\mapsto \tr_{\S}(M)$ is in fact functorial:
 
\begin{lemma}\label{lemma_traces_are_functorial}
Let $\A$ be a preadditive category and $\mathcal{S}$ a class of $\mathcal{A}$-modules. Then, the assignment 
\[
M\mapsto\tr_\mathcal{S}(M)
\] 
defines a subfunctor of the identity $\mod \A\to\mod \A$.
\end{lemma}
\begin{proof}
It is enough to show that, given a morphism $\phi \colon M\to N$ in 
$\mod\mathcal{A}$, then $\phi 
(\tr_\mathcal{S}(M))\leq\tr_\mathcal{S}(M)$. But this is 
clear since, given any morphism $f\colon S\to M$, with 
$S\in\mathcal{S}$, we have $\phi (\Im(f))=\Im(\phi\circ f)$.
%
%
%given a morphism $\phi\colon M\to N$ in $\mod \A$, then $\phi(\tr_\S(M))\leq \tr_{\S}(N)$. But this clear since 
%\[
%\phi(\tr_{\S}(M))=\sum_{S\in \S}\   \sum_{f\colon S\to M}\Im(\phi\circ f)\overset{(*)}{\leq} \sum_{S\in \S}\   \sum_{g\colon S\to N}\Im(g)=\tr_{\S}(N), 
%\]
%where $(*)$ holds since, given $f\colon S\to M$, $\phi\circ f$ is a morphism $S\to N$.
\end{proof}

Recall now that an $\mathcal{A}$-{\bf bimodule} is a bifunctor 
\[
X\colon\mathcal{A}^{\op}\times\mathcal{A}\to \Ab
\] 
which is additive in each component. The {\bf regular $\A$-bimodule} is the bifunctor $\A(-,-)\colon \mathcal{A}^{\op}\times\mathcal{A}\to \Ab$.  
A {\bf sub-$\A$-bimodule} $Y$ of an $\A$-bimodule $X$ is just a sub-bifunctor, that is, $Y$ is a bifunctor $\mathcal{A}^{\op}\times\mathcal{A}\to \Ab$ such that 
\begin{itemize}
\item $Y(a,b)\leq X(a,b)$ is a subgroup, for all $a,\, b\in\Ob(\A)$;
\item $Y(\alpha,\beta)$ is the restriction of $X(\alpha,\beta)$, for each morphism $(\alpha,\beta)$ in $\mathcal{A}^{\op}\times\mathcal{A}$.
\end{itemize}
Given an $\A$-bimodule $X$ and $b\in\Ob(\mathcal{A})$, we can define two additive functors 
\[
X_b\colon\mathcal{A}^{\op}\to\Ab\quad\text{and}\quad X^b\colon \A\to \Ab, \text{ where:}
\] 
\begin{itemize}
\item $X_b\colon a\mapsto X(a,b)$ and $X^b\colon a\mapsto X(b,a)$, for all $a\in \Ob(\A)$;
\item $X_b(f):=(-\circ f)\colon X(a',b)\to X(a,b)$ and $X^b(f):=(f\circ -)\colon X(b,a)\to X(b,a')$, for all $f\colon a\to a'$ in $\A$.
\end{itemize}
Hence, $X_b\in\mod \A$ and $X^a\in\lmod \A:=\mod {\A^{\op}}$. One can check that the assignment $b\mapsto X_b$ (resp., $a\mapsto X^a$) defines an additive functor $\mathcal{A}\to\mod \A$ (resp., $\mathcal{A}^{\op}\to\lmod {\A}$). Notice that, applying this construction to the regular $\A$-bimodule we get $\A(-,-)_a=H_a$, for all $a\in\Ob(\A)$. 

\medskip
On the other hand, an additive functor $M\colon\mathcal{A}\to\mod \A$ (resp.,  $L\colon\mathcal{A}^{\op}\to\lmod {\A}$) defines an $\mathcal{A}$-bimodule 
\[
X_M\colon\mathcal{A}^{\op}\times\mathcal{A}\to\Ab\qquad\text{(resp., $_LX \colon\mathcal{A}^{\op}\times\mathcal{A}\to\Ab$)},\text{ where}
\] 
\begin{itemize}
\item given $a,\, b\in \Ob(\A)$, $X_M(a,b):=(M(b))(a)$ and ${}_LX(a,b):=(L(a))(b)$; 
\item given $(f,g)\colon (a,b)\to (a',b')$ in $\mathcal{A}^{\op}\times\mathcal{A}$, we let $X_M(f,g):=(M(b'))(f)\circ M(g)_a$, while ${}_LX(f,g):=L(f)_b\circ (L(a'))(g)$.
\end{itemize}
This allows us to see $\mathcal{A}$-bimodules either as functors $\mathcal{A}\to\mod \A$ or $\mathcal{A}^{\op}\to\lmod {\A}$.

\begin{lemma} \label{lemma_traces_are_functorial_bimodules}
Let $\A$ be a preadditive category, $X$ an $\mathcal{A}$-bimodule and $\mathcal{S}$ a class of right $\mathcal{A}$-modules. The assignment $b\mapsto \tr_\mathcal{S}(X_b)$ defines a subfunctor of the functor $b\mapsto X_b$ defined in the above discussion. The associated $\mathcal{A}$-bimodule, denoted by $\tr_\mathcal{S}(X)$, is then a sub-$\mathcal{A}$-bimodule of $X$.
 \end{lemma}
\begin{proof}
If $\beta\colon b\rightarrow b'$ is a morphism in $\mathcal{A}$, then $X_\beta:=X(-,\beta)\colon X_b\to  X_{b'}$ is a morphism in $\mod \A$. Since $\tr_\mathcal{S}\colon\mod \A\to \mod \A$ is a subfunctor of the identity, it follows that $X_\beta (\tr_\mathcal{S}(X_b))\leq\tr_\mathcal{S}(X_{b'})$, so we get an induced morphism $\tr_\mathcal{S}(X_\beta )\colon\tr_\mathcal{S}(X_b)\to  \tr_\mathcal{S}(X_{b'})$ in $\mod \A$. We define $\tr_\mathcal{S}(X_\beta)$ to be the image of $\beta$ by the desired functor $\mathcal{A}\to \mod \A$. The rest of the proof is routine. 
\end{proof}

Given an $\A$-bimodule $X$ and a class of modules $\S$, the $\A$-bimodule $\tr_\S(X)$ is called the {\bf trace of $\mathcal{S}$ on the $\mathcal{A}$-bimodule $X$}.

\subsection{Locally coherent categories of modules}\label{subs_locally_coh}

Given a  small preadditive category $\A$, a right $\A$-module $M$ is said to be {\bf finitely presented} if the functor $\hom_\A(M,-)\colon\mod \A\to \Ab$ commutes with direct limits. As a consequence of Coro.\,\ref{application_yoneda}, one can deduce (exactly as one does for categories of modules over a unitary ring) that the category $\mod \A$ is  {\bf locally finitely presented}, that is, any right $\A$-module can be written as a direct limit of finitely presented modules. In what follows we go one step further and characterize those categories $\A$ for which $\mod \A$ is also locally coherent, that is, we give a necessary and sufficient condition for $\fpmod \A$ (the category of finitely presented modules) to be closed under taking kernels in $\mod \A$. 

\smallskip
Recall that, given a preadditive category $\C$ and a morphism $\phi\colon X\to Y$ in $\C$, a morphism $\psi\colon K\to X$ in $\C$ is said to be a {\bf pseudo-kernel} of $\phi$ if, for any $Z\in \Ob(\C)$, the following sequence of Abelian groups is exact:
\[
\xymatrix@C=15pt{
{\C}(Z,K)\ar[rr]^{(Z,\psi)}&&{\C}(Z,X)\ar[rr]^{(Z,\phi)}&&{\C}(Z,Y).
} 
\]
{\bf Pseudo-cokernels} are defined dually. Let us remark that any Abelian or triangulated category has pseudo-kernels and pseudo-cokernels.  Pseudo-kernels have been introduced, under the name of ``weak kernels'', by Freyd \cite{Freyd}.
%%of $f$ is a morphism  (resp., $h\colon b\to c$) such that the sequence of Abelian groups 
%%\[
%%\xymatrix@C=15pt{
%%{\A}(x,k)\ar[rr]^{(x,g)}&&{\A}(x,a)\ar[rr]^{(x,f)}&&{\A}(x,b)
%%} 
%%\quad
%%\left(\text{
%%resp.,
%%$\xymatrix@C=15pt{
%%{\A}(c,x)\ar[rr]^{(h,x)}&&{\A}(b,x)\ar[rr]^{(f,x)}&&{\A}(a,x)
%%}$ 
%%}\right)
%%\] 
%%is exact, for all $x\in\Ob(\A)$. 
% 

\begin{corollary} \label{lem.locally coherent module categories}
Let $\A$ be a  small preadditive category. The following  are equivalent:
\begin{enumerate}[\rm (1)]
\item $\mod{\A}$ is a locally coherent Grothendieck category;
\item the subcategory $\proj(\A)$($\cong \widehat\A_{\oplus}$) of $\mod{\A}$ has pseudo-kernels;
\item the additive closure $\widehat{\A}$ has pseudo-kernels. 
\end{enumerate}
\end{corollary}
\begin{proof}
(1)$\Rightarrow$(2). Let $\phi\colon P\to Q$ be morphism in $\proj(\A)$. Being $\mod{\A}$ locally coherent, $\Ker(\phi)\in\fpmod\A$. Consider an epimorphism $\pi\colon \coprod_{i=1}^nH_{a_i}\to \Ker(\phi)$, with $n\in \N$ and $a_i\in \Ob(\A)$; it is not difficult to prove that the composition $\psi\colon \coprod_{i=1}^nH_{a_i}\to \Ker(\phi)\to P$ is a pseudo-kernel of $\phi$ in $\proj(\A)$.

\smallskip\noindent
(2)$\Rightarrow$(3). By Coro.\,\ref{application_yoneda} we can identify $\widehat{\A}$ with the full subcategory of $\proj(\A)$ of the objects of the form $\coprod_{i=1}^nH_{a_i}$, with $n\in \N$ and $a_i\in \Ob(\A)$. Consider then a morphism $\phi\colon \coprod_{i=1}^n H_{a_i}\to \coprod_{j=1}^m H_{b_j}$ and a pseudo-kernel $\psi\colon K\to \coprod_{i=1}^n H_{a_i}$ of $\phi$ in $\proj(\A)$. Take  an epimorphism $\pi\colon \coprod_{l=1}^{n'}H_{c_l}\to K$, with $n'\in \N$ and $c_l\in \Ob(\A)$; it is not difficult to prove that the composition $\psi\colon \coprod_{l=1}^{n'}H_{c_l}\to K\to\coprod_{i=1}^nH_{a_i}$ is a pseudo-kernel in $\widehat\A$.

\smallskip\noindent
(3)$\Rightarrow$(1) can be proved as in \cite[Lem.\,1.4.5]{Freyd}.
\end{proof}

\subsection{The center of a preadditive category}\label{subs_centre}

Recall the following definition from \cite{Ga}: the {\bf center} $Z(\A)$ of a preadditive category $\A$ is the ring of self-natural transformations of the identity functor $\id_{\A}$, that is, 
\[
Z(\A):=(\Padd(\A,\A))(\id_\A,\id_\A),
\]
where the above formula just means that, in the $2$-category $\Padd$, we consider the category of endomorphisms $\Padd(\A,\A)$ and, in that category, we take the endomorphism ring of the object $\id_\A$.

\smallskip
It is an exercise on the definitions to verify that, given a unitary ring $R$, the commutative ring $Z(R)$ is isomorphic to the subring $\{r\in R:rs=sr,\,\forall s\in R\}$, which is usually called the center of $R$. On the other hand, given a small preadditive category $\A$, we can consider both $Z(\A)$ and $Z(\mod \A)$. In the following proposition we show that both choices give the same ring:

\begin{proposition} \label{prop.Z(A)-vs-Z(Mod-A)}
Given a small preadditive category $\A$, $Z(\A)\cong Z(\mod \A)$.
\end{proposition}
\begin{proof}
Consider the following maps:
\[
\Phi\colon Z(\A)\to Z(\mod \A)\quad\text{and}\quad\Psi\colon Z(\mod \A)\to Z(\A),
\]
such that, given $\alpha\colon\id_\A\to \id_\A$ and $M\in \mod \A$, we define $\Phi(\alpha)_M\colon M\to M$ as follows: $\Phi(\alpha)_{M,a}:=M(\alpha_a)\colon M(a)\to M(a)$, for all $a\in \Ob(\A)$. On the other hand, given $\beta\colon \id_{\mod \A}\to \id_{\mod \A}$ and $a\in \A$, we let $\Psi(\beta)_a\colon a\to a$ be the unique morphism such that $\A(-,\Psi(\beta)_a)=\beta_{H_a}\colon H_a\to H_a$. It is now an easy exercise to verify that $\Phi$ and $\Psi$ are each other inverse.
\end{proof}

As a consequence of the above proposition, one obtains that $Z(\A)\cong Z(\widehat{\A})\cong Z(\widehat{\A}_\oplus)$. To see this, one can use that $\A$, $\widehat \A$ and $\widehat{\A}_\oplus$ are Morita equivalent categories. More generally, this result shows that the center of a small preadditive category is invariant under Morita equivalence.

%We have an obvious additive {\bf projection functor} $p:\A\to\A/t(\A)$, and the functor $p_*:\mod{\A}/t(\A)\to\mod{\A}$ ($M\mapsto  M\circ p$), usually called the {\bf restriction of scalars}, is fully faithful and has both a left adjoint $p^*:\mod{\A}\to\mod{\A}/t(\A)$ and a right adjoint $p^!:\mod{\A}\to\text{Mod}-\/t(\A)$.  Note that, for each $x\in\Ob(\A)$, one has $p^*(\A(-,x))=(\A/t(\A))(-,x))$. If $\eta :1_{Mod-\A}\to p_*\circ p^*$ is the unit of the adjunction $(p^*,p_*)$, it easily follows that $\eta_{\A(-,x)}:\A(-,x)\to (p_*\circ p^*)(\A(-,x))=p_*((\A/t(\A))(-,x))=(1:t)((\A(-,x))$ is the canonical projection, for all $x\in\Ob(\A)$. It follows that $\eta_P$ is an epimorphism aand  $\Ker(\eta_P)=t(P)$, for each projective $\A$-module $P$. Bearing in mind that $p^*$ and $p_*$ are both right exact, we readily see that $\eta_M$ is an epimorphism and that $\Ker(\eta_M)\in\T$, and hence $\Ker(\eta_M)\subseteq t(M)$,  for all $M\in\mod{\A}$. 
%Recall that, by properties of adjunction,  the essential image of $p_*$ consists of the $\A$-modules $X$ such that  $\eta_X:X\to (p_*\circ p^*)(X)$ is an isomorphism. We will frequently identify $\mod{\A}/t(\A)$ with this last subcategory of $\mod{\A}$. Note that we clearly have $\F\subseteq\Im(p_*)\cong\mod{\A}/t(\A)$. This allows us to view $\F$ as torsionfree class in $\mod{\A}/t(\A)$, the associated torsion pair in this latter category being $\mathbf{t}'=(\T\cap\mod{\A}/t(\A),\F)$

\section{Ideals of preadditive categories}\label{subs_tors_and_id}

The section is organized as follows: we start recalling the definition and some basic facts about (two-sided) ideals of a preadditive categories in Sec.\,\ref{Ideals and quotient categories}; we then specialize the discussion to idempotent ideals in Sec.\ref{Idempotent ideals and traces of projectives} and, between them, we give some equivalent characterizations of idempotent ideals that are traces of projective modules; finally, we show in Sec.\,\ref{subs_directsummands} that the direct sum decompositions of the regular bimodule are all induced by central idempotents of $\A$.

\subsection{Ideals and quotient categories}\label{Ideals and quotient categories}

Let $\A$ be a  small preadditive category, a ({\bf two-sided}) {\bf ideal} of $\A$ is a subfunctor 
\[
\I(-,-)\leq \A(-,-)\colon \A^{\op}\times \A\to \Ab.
\] 
That is, given $a,\,a',\,b,\,b'\in \Ob(\A)$, and $f\in \I(a,b)$, $r\in \A(a',a)$ and $l\in \A(b,b')$, the composition $l\circ f\circ r$ belongs in the subgroup $\I(a',b')\leq \A(a',b')$. Notice that one can equivalently describe ideals of $\A$ as sub-$\A$-bimodules of the regular bimodule $\A(-,-)$.

Given an ideal $\I$ of $\A$, we can form a new {\bf quotient category} $\A/\I$ with objects $\Ob(\A/\I)=\Ob(\A)$ and morphisms defined by
\[
(\A/\I)(a,b):=\A(a,b)/\I(a,b),
\] 
for all $a,\,b\in \Ob(\A)$, with identities and composition law induced by the ones of $\A$. Of course there is a natural functor $\pi_\I\colon \A \to \A/\I$, that induces a restriction of scalars:
\[
(\pi_\I)_*\colon \mod{\A/\I}\to \mod \A\qquad\text{such that}\qquad M\mapsto M\circ\pi_\I.
\]
As for the case when $\A$ is a ring, one can give a very explicit description of the extension of scalars $(\pi_\I)^*$ (the left adjoint to $(\pi_\I)_*$). Indeed, given $M\in\mod {\A}$, we define a subfunctor $M\I\colon \A^{\op}\to \Ab$ of $M$ such that
\[
M\I(a):=\sum_{b\in\Ob(\A),\,\alpha\in\I(a,b)}\Im(M(\alpha)).
\]
Then, $(\pi_\I)^*(M)\cong M/M\I$. Let us remark that, by construction, $H_a\I=\I(-,a)$ for each $a\in\Ob(\A)$, so
\begin{equation}\label{extending_representables_eq}
(\pi_\I)^*(H_a)\cong H_a/H_a\I\cong H_a/\I(-,a).
\end{equation}

\begin{lemma}
Let $\A$ be a  small preadditive category and let $\I$ be an ideal. The class of right $\A$-modules $M$ such that $M\I=0$ coincides with 
\[
\Gen\{H_a/\I(-,a):a\in\Ob(\A)\},
\] 
that is, with those objects that can be written as a quotient of a coproduct of modules, each isomorphic to some of the $H_a/\I(-,a)$'s. Furthermore, the full subcategory $\Gen\{H_a/\I(-,a):a\in\Ob(\A)\}$ of $\mod \A$ is equivalent to $\mod \A/\I$.
\end{lemma}
\begin{proof}
Let us start by noticing that $\left(H_a/\I(-,a)\right)\I=0$ for all $a\in \Ob(\A)$, so that, if $M\in \Gen\{H_a/\I(-,a):a\in\Ob(\A)\}$, then $M\I=0$. On the other hand, suppose $M\I=0$ and consider an epimorphism 
\[
\phi=(\phi_i)_{I}\colon\coprod_{i\in I}H_{a_i}\to M\to 0.
\] 
Since $M\I=0$, each $\phi_i$ factors as $\phi_i=\psi_i\pi_i$, where $\pi_i\colon H_{a_i}\to H_{a_i}/\I(-,a_i)$ is the  obvious projection. Then the epimorphism 
\[
\psi:=(\psi_i)_{I}\colon\coprod_{i\in I}H_{a_i}/\I(-,a_i)\to M
\] 
shows that $M\in \Gen\{H_a/\I(-,a):a\in\Ob(\A)\}$. 
\end{proof}

Given a unitary ring, one can always construct the two-sided ideal generated by a given family of elements. The analogous construction in the more general setting of preadditive categories is given in the following definition:

\begin{definition} \label{def.principal ideal}
Let $\alpha \colon x\to  y$ be a morphism in $\A$ and let $\mathcal{M}$ be a set of morphisms. Then define:
\begin{itemize}
\item the ({\bf two-sided}) {\bf ideal of $\mathcal{A}$ generated by $\alpha$}, denoted by $\mathcal{A}\alpha\mathcal{A}\colon \A^{\op}\times \A\to \Ab$, is the subfunctor of $\A(-,-)$ such that $(\mathcal{A}\alpha\mathcal{A})(a,b)$ is the subgroup of $\mathcal{A}(a,b)$ generated by compositions $\delta\circ\alpha\circ\gamma$, where $\gamma\in\mathcal{A}(a,x)$ and $\delta\in\mathcal{A}(y,b)$;
\item the ({\bf two-sided}) {\bf ideal of $\mathcal{A}$ generated by $\mathcal{M}$}, denoted by $\mathcal{A}\mathcal{M}\mathcal{A}\colon \A^{\op}\times \A\to \Ab$, is the sum 
\[
\mathcal{A}\mathcal{M}\mathcal{A}:=\sum_{\mu\in\mathcal{M}}\mathcal{A}\mu\mathcal{A}.
\] 
That  is, $(\mathcal{A}\mathcal{M}\mathcal{A})(a,b)=\sum_{\mu\in\mathcal{M}}(\mathcal{A}\mu\mathcal{A})(a,b)$, where the sum in the second member of the last equality is the sum of subgroups of the Abelian group $\mathcal{A}(a,b)$. 
\end{itemize}
\end{definition}
For a set of morphisms $\M$ in $\A$, a general element in $(\mathcal{A}\mathcal{M}\mathcal{A})(a,b)$ is a finite sum:
\begin{equation}\label{general_form_of_ideal_generated}
\xymatrix@R=-5pt@C=12pt{
&&&&x_1\ar[rrrr]^{m_1}&&&&y_1\ar[ddrrrr]^{\delta_1}\\
&&&&&&+\\
a\ar[uurrrr]^{\gamma_1}\ar[ddrrrr]_{\gamma_k}&&&&&&\vdots&&&&&&b\\
&&&&&&+\\
&&&&x_k\ar[rrrr]_{m_k}&&&&y_k\ar[uurrrr]_{\delta_k}
}
\end{equation}
where $m_1,\dots,m_k\in \M$. 

\medskip
In the following lemma we study the invariance of the set of two-sided ideals in $\A$ under Morita equivalence:

\begin{lemma}
Let $\A$ and $\B$ be two Morita equivalent small preadditive categories. Then,
\begin{enumerate}[\rm (1)]
\item there is a bijection between the ideals of $\A$ and those of its Cauchy completion $\widehat \A_{\oplus}$;
\item there is a bijection between the sets of ideals of $\A$ and $\B$.
\end{enumerate}
\end{lemma}
\begin{proof}
(1) Consider the canonical inclusion $\A\to \widehat \A_{\oplus}$ and identify $\A$ as a full subcategory of $\widehat\A_\oplus$. Then there is a map
\[
F\colon\{\text{Ideals of $\A$}\}\to \{\text{Ideals of $\widehat\A_\oplus$}\},\qquad I\mapsto \widehat\A_\oplus I\widehat\A_\oplus,
\]
where $\widehat\A_\oplus I\widehat\A_\oplus$ denotes the ideal of $\widehat\A_{\oplus}$ generated by the union $\bigcup_{a,b\in\Ob(\A)}I(a,b)$. On the other hand, one can construct a map in the opposite direction as follows:
\[
G\colon\{\text{Ideals of $\widehat\A_\oplus$}\}\to \{\text{Ideals of $\A$}\},\qquad J\mapsto J\restriction_{\A^{\op}\times\A}.
\]
It is now routine to check that these maps are each other inverse. 

\smallskip\noindent
(2) This is an application of part (1), just using the fact that $\A$ and $\B$ are Morita equivalent if and only if $\widehat \A_{\oplus}$ is equivalent to $\widehat \B_{\oplus}$.
\end{proof}

%
%\begin{example}
%Let $\alpha \colon x\to  y$ be a morphism in $\A$. The {\bf (principal) right  ideal of $\mathcal{A}$ generated by $\alpha$}, denoted $\alpha\mathcal{A}$,  is the image of the morphism $H_\alpha :H_x=\mathcal{A}(-,x)\to  \mathcal{A}(-,y)=H_y$.  Then it is the submodule of $H_y$ given on objects by $$(\alpha\mathcal{A})(a)=\{\beta\in\mathcal{A}(a,y):\beta =\alpha\circ\gamma\text{, for some }\gamma\in\mathcal{A}(a,x)\}$$ 
%
%Let $\alpha \colon x\to  y$ be a morphism in $\A$. The \emph{(two-sided) ideal of $\mathcal{A}$ generated by $\alpha$}, denoted $\mathcal{A}\alpha\mathcal{A}$,  will be the one given on objects by the following rule:
%
%  $(\mathcal{A}\alpha\mathcal{A})(a,b)$ is the subgroup of $\mathcal{A}(a,b)$ consisting of the $\beta\in\mathcal{A}(a,b)$ such that $\beta$ is a finite sum of compositions $\delta_i\circ\alpha\circ\gamma_i$, where $\gamma_i\in\mathcal{A}(a,x)$ and $\delta_i\in\mathcal{A}(y,b)$ for all $i$.   
%\end{example}

\subsection{Idempotent ideals and traces of projectives}\label{Idempotent ideals and traces of projectives}

Given two ideals $\I$ and $\mathcal J$ of $\A$, we define a new ideal $\I\cdot \mathcal J$ as follows:
\[
(\I\cdot \mathcal J)(a,b):=\left\{\sum_{i=1}^n\phi_i\circ \psi_i: \phi_i\in \I(c_i,b),\, \psi_i\in \mathcal J(a,c_i)\right\}.
\]
An ideal $\I$ is said to be {\bf idempotent} if $\I\cdot\I=\I$.

\begin{lemma}
Let $\A$ be a small preadditive category and let $\I$ be an idempotent ideal. The class of right $\A$-modules $M$ such that $M\I=M$ coincides with 
\[
\Gen\{\I(-,a):a\in\Ob(\A)\}.  
\]
That is, with those objects that can be written as a quotient of a coproduct of modules, each isomorphic to some of the $\I(-,a)$'s. 
\end{lemma}
\begin{proof}
If $M\in \Gen\{\I(-,a):a\in\Ob(\A)\}$, then  $M\I=M$. On the other hand, suppose $M\I=M$, consider an epimorphism 
$\coprod_{i\in I}H_{a_i}\to M$ 
and let $K:=\Ker(\phi)$. Then
\[
M=M\I\cong \frac{\coprod_{i\in I}H_{a_i}}{K}\I\cong \frac{\coprod_{i\in I}H_{a_i}\I}{K\cap \coprod_{i\in I}H_{a_i}\I}\cong \frac{\coprod_{i\in I}\I(-,{a_i})}{K\cap \coprod_{i\in I}\I(-,{a_i})},
\]
and this last module clearly belongs to $\Gen\{\I(-,a):a\in\Ob(\A)\}$.
\end{proof}

Notice that Lem.\,\ref{lemma_traces_are_functorial_bimodules} applies in particular to the regular bimodule $\mathcal{A}(-,-)$. Hence, given a class of right $\A$-modules $\S$, the trace of $\mathcal{S}$ in the regular bimodule, called the {\bf trace of $\mathcal{S}$ in $\mathcal{A}$}, and denoted by $\tr_\mathcal{S}(\mathcal{A})$, is a two-sided ideal of  $\mathcal{A}$. As for modules over a unital ring, the situation when $\mathcal{S}$ consists of projective $\mathcal{A}$-modules deserves a special attention:

\begin{lemma} \label{lem.trace-of-projectives is idempotent}
Let $\A$ be a small preadditive category and let $\S$ be a class of right $\mathcal{A}$-modules. Then,
\begin{enumerate}[\rm (1)]
\item there is a (small) subset $\S'$ of $\S$ such that $\tr_\mathcal{S}(\mathcal{A})=\tr_{T}(\mathcal{A})$, where $T:=\coprod_{\S'}S$;
\item for each $a,\, b\in\Ob(\A)$, the group $\tr_T(\mathcal{A})(a,b)$ consists of the morphisms $\gamma\in\A(a,b)$ such that the map $H_\gamma\colon H_a\to H_b$ factors through a finite coproduct of copies of $T$;
%can be factored through a finite coproduct of objects in $\P$, that is, 
%
%$\gamma\colon a\to b$ in $\A$ such that the induced map $H_\gamma\colon H_a\to H_b$ can be factored through a finite coproduct of objects in $\P$, that is, there exist $P_1,\dots,P_k\in \P$ and a factorization 
%\[
%\xymatrix@R=-3pt@C=12pt{
%&&&&&&P_1\ar[ddrrrr]^{\beta_1}\\
%&&&&&&+\\
%H_a\ar[rr]^(.35){H_\gamma}&&H_b\ \ =
%\ \  H_a\ar[uurrrr]^{\alpha_1}\ar[ddrrrr]_{\alpha_k}&&&&\vdots&&&&H_b\\
%&&&&&&+\\
%&&&&&&P_k\ar[uurrrr]_{\beta_k}
%}
%\]
\item if each $S$ in $\S$ is projective, then the ideal $\tr_\mathcal{S}(\mathcal{A})$ is idempotent.
\end{enumerate}
Summarizing, if $\S$ is a class of projectives, then $\tr_\mathcal{S}(\mathcal{A})$ is an idempotent ideal such that, for each $a,\,b\in\Ob(\A)$, $\tr_\mathcal{S}(\mathcal{A})(a,b)$ consists of those $\gamma\in\A(a,b)$ such that the map $H_\gamma$ factors through a finite coproduct of objects in $\S$.
\end{lemma}
\begin{proof}
(1) Given $b\in\Ob(\mathcal{A})$, the submodules  $L$ of $H_b:=\mathcal{A}(-,b)$ which are of the form $L=\Im(f)$, for some morphism $f\colon S\to  H_b$, with $S\in\mathcal{S}$, form a set. Therefore we can select a set $\mathcal{S}_b\subseteq\mathcal{S}$ such that $\tr_\mathcal{S}(H_b)=\tr_{\mathcal{S}_b}(H_b)$. Letting  $\mathcal{S}':=\bigcup_{b\in\Ob(\mathcal{A})}\mathcal{S}_b$, it is clear that $\tr_{\mathcal{S}'}(H_b)=\tr_\mathcal{S}(H_b)$, for all $b\in\Ob(\mathcal{A})$. %This means that the functor $\mathcal{A}\to \mod \A$, acting on objects as $b\mapsto \tr_\mathcal{P}(H_b)=\tr_\mathcal{P}(\mathcal{A}(-,b))$ (see Lem.\,\ref{lemma_traces_are_functorial_bimodules}), can be defined using $\mathcal{P}'$ instead of $\mathcal{P}$. 
Hence, if we put $T:=\coprod_{S\in\mathcal{S}'}S$, then $\tr_\mathcal{S}(H_b)=\tr_\mathcal{S'}(H_b)=\tr_T(H_b)$, for each $b\in\Ob(\mathcal{A})$. 

\medskip\noindent
(2) Let us denote by $\I_T(-,-)\leq \A(-,-)$ the ideal of morphism $\gamma$ such that $H_\gamma$ factors through a finite coproduct of copies of $T$. We then have to verify that, given $a,\,b\in\Ob(\A)$, $\I_T(a,b)= \tr_T(\A)(a,b)=\tr_T(H_b)(a)$. Indeed, consider a morphism $\gamma\colon a\to b$ in $\A$, and suppose that $H_\gamma=\sum_{i=1}^kg_if_i$, where $f_i$ and $g_i$ are as in the following picture:
\begin{equation}\label{picture}
\xymatrix@R=-4pt@C=18pt{
&&&&T\ar[ddrrrr]^{g_1}\\
&&&&+\\
\sum_{i=1}^kg_if_i\colon H_a\ar[uurrrr]^(.6){f_1}\ar[ddrrrr]_(.6){f_k}&&&&\vdots&&&&H_b\\
&&&&+\\
&&&&T\ar[uurrrr]_{g_k}
}
\end{equation}
In particular, by the Yoneda Lemma, we get  $\gamma=(H_\gamma)_a(\id_a)=\sum_{i=1}^k(g_i)_a(f_i)_a(\id_a)\leq \sum_{i=1}^k\Im(g_i)_a\leq \tr_{T}(\A)(a,b)$.\\
On the other hand, given an element $\delta\in \tr_{T}(\A)(a,b)$, by definition of trace, there exist $g_1,\dots,g_k\colon T\to H_b$, such that $\delta\in \sum_{i=1}^k\Im(g_i)_a$. This means that there exist $x_1,\dots,x_k\in T_a$, such that $\delta= \sum_{i=1}^k(g_i)_a(x_i)$. Again, by the Yoneda Lemma, there exist unique morphisms $f_1,\dots,f_k\colon H_a\to T$ such that $(f_i)_a(\id_a)=x_i$. Hence, we are again in the situation of \eqref{picture}, so that $H_\delta$ factors through $T^k$, since $H_\delta=\sum_{i=1}^kg_if_i$.

\smallskip\noindent
(3) By part (1), it is clear that we can choose a projective module $T$ such that $\tr_\S(-,-)=\tr_T(-,-)$. Consider  an epimorphism $p\colon \coprod_{I}H_{x_i}\to T$, with $x_i\in \Ob(\A)$ for all $i\in I$ and, using the projectivity of $T$, choose a section $u\colon T\to \coprod_{I}H_{x_i}$, that is $p\circ u=\id_T$. Now, given $\alpha\in\mathcal{A}(a,b)$ such that  $H_\alpha$ factors through $T$, that is, we have  $f\colon H_a\to T$ and $g\colon T\to H_b$ such that $H_\alpha=g\circ f$, there exists a finite subset $F\subseteq I$ such that the map  $u\circ f\colon H_a\to \coprod_{I}H_{x_i}$ factors through the inclusion $\iota_F\colon \coprod_{F}H_{x_i}\to \coprod_{I}H_{x_i}$; let also $\pi_F\colon \coprod_{I}H_{x_i}\to \coprod_{F}H_{x_i}$ be the corresponding projection. We obtain a commutative diagram as follows:
\[
\xymatrix@R=15pt@C=30pt{
H_a\ar@/_-20pt/[rrrrrr]|{H_\alpha}\ar@/_+5pt/[dr]_f\ar[rr]&&\coprod_{I}H_{x_i}\ar[r]^{\pi_F}& \coprod_{F}H_{x_i}\ar[r]^{\iota_F}&\coprod_{I}H_{x_i}\ar@/_+5pt/[dr]_p\ar[rr]&&H_b\\
&T\ar@/_+5pt/[ur]_u&&&&T\ar@/_+5pt/[ur]_g
}
\]
Now, let also $\pi_k\colon \coprod_{I}H_{x_i}\to H_{x_k}$ and $\iota_k\colon H_{x_k}\to \coprod_{I}H_{x_i}$ be the obvious projection and inclusion, for $k\in F$. Using the Yoneda Lemma, the above discussion shows that 
\begin{align*}
\alpha&=(H_\alpha)_a(\id_a)\\
&=\sum_{F}((g\circ p\circ \iota_k)\circ(\pi_k\circ u\circ f))_a(\id_a)\\
&=\sum_{F}(g\circ p\circ \iota_k)_{x_k}(\id_{x_k})\circ(\pi_k\circ u\circ f)_a(\id_a).
\end{align*}
Now, both $(\pi_k\circ u\circ f)_a(\id_a) \in \tr_T(\A)(a,x_k)$ and $(g\circ p\circ \iota_k)_{x_k}(\id_{x_k})\in \tr_T(\A)(x_k,b)$, showing that $\alpha\in \tr_{T}(\A)^2(a,b)$, as desired.
%
%This shows that $H_\alpha$ can be written as a sum of $|F|$-many morphism morphism 
%
%Hence, $H_\alpha$ is a composition of $\pi_F\circ u\circ f$ and $g\circ p\circ \iota_F$, and each of these two maps is a sum of $|F|$-many maps between representables that factor through $T$. Hence both $\pi_F\circ u\circ f$ and $g\circ p\circ \iota_F$ belong in $\tr_T(-,-)$
%
%
%
%admits a factorization $H_a\stackrel{f}{\to }Q\stackrel{g}{\to }H_b$, then $\alpha\in\mathcal{I}_Q^2(a,b)$. Using the equality $p\circ u=1_Q$, we have that $\alpha =(g\circ p)\circ (u\circ f)$. But $u\circ f:H_a\to \coprod_{i\in I}H_{b_i}$ factors in the form $H_a\stackrel{f'}{\to }\coprod_{i\in F}H_{b_i}\hookrightarrow\coprod_{i\in I}H_{b_i}$, for some finite subset $F\subseteq I$, because $H_a$ is a finitely generated $\mathcal{A}$-module. It then follows that $H_\alpha$ factors as a composition $H_a\to \coprod_{i\in F}^nH_{b_i}\to  H_b$, where both morphisms in the composition factors through $Q$. But such a composition is a sum of the compositions of the form $H_a\to  H_{b_i}\to  H_b$, where each of the two morphism factors through $Q$. Then each of these last compositions is in $\mathcal{I}_Q^2(a,b)$, and therefore $\alpha\in\mathcal{I}_Q^2(a,b)$. 
\end{proof}

Given a set $\M$ of morphisms in $\A$ we have described in \eqref{general_form_of_ideal_generated} the general form of an element in the ideal $\A\M\A$.  Of particular interest for us is the case when $\mathcal{M}=\mathcal{E}$ is a set of idempotent endomorphisms of objects of $\mathcal{A}$. As in the case of modules over unital rings, the reason is the following proposition:

\begin{proposition} \label{prop.ideals generated by idempotents}
Let $\mathcal{A}$ be a preadditive category and let $\mathcal{E}$ be a set of idempotent endomorphisms of objects of $\mathcal{A}$.
Then, letting $\mathcal{P}_\mathcal{E}:=\{\epsilon\mathcal{A}:\epsilon\in\mathcal{E}\}$ (where each $\epsilon\mathcal{A}$ is a finitely generated projective module, see Lem.\,\ref{description_fpp_lemma}), 
\[
\tr_{\P_\E}(\A)=\A\E\A.
\]
% in  $\mathcal{A}$ is the ideal $\mathcal{A}\mathcal{E}\mathcal{A}$. 
\end{proposition}
\begin{proof}
Let us start by noting that $\tr_{\P_\E}(\A)=\sum_{\epsilon\in \E}\tr_{\epsilon\A}(\A)$ and $\A\E\A=\sum_{\epsilon\in \E}\A\epsilon\A$. Hence, our statement will follow if we prove that, for a given idempotent endomorphism $\epsilon\colon x\to x$ in $\A$, $\tr_{\epsilon \A}(\A)=\A\epsilon\A$.
%
%Let $\mathcal{P}_\mathcal{E}:=\{\epsilon\mathcal{A}: \epsilon\in\mathcal{E}\}$, and let us verify that $\tr_{\P_\E}(\A)=\A\E\A$. The trace of $\mathcal{P}_\mathcal{E}$ in $\mathcal{A}$ is the sum of the traces of the $\epsilon\mathcal{A}$ in $\mathcal{A}$, with $\epsilon\in\mathcal{E}$. So the task is reduced to prove that if $\epsilon=\epsilon^2\in\mathcal{A}(x,x)$ is an idempotent endomorphism, then  $\tr_{\epsilon\mathcal{A}}(\mathcal{A})=\mathcal{A}\epsilon\mathcal{A}$. 

\smallskip
For the inclusion ``$\geq$'', fix $a,\, b\in\Ob(\mathcal{A})$ and take any composition 
\[
\xymatrix@C=12pt{
a\ar[rr]^{\alpha}&&x\ar[rr]^{\epsilon}&&x\ar[rr]^{\beta}&&b
}
\] 
of morphisms in $\mathcal{A}$. If we denote by $H'_\beta$ the restriction of $H_\beta \colon H_x\to  H_b$ to the submodule $\epsilon\mathcal{A}\leq H_x$, then $(H'_\beta )_a\colon(\epsilon\mathcal{A})(a)\to H_b(a)$ maps $\epsilon\circ\alpha$ onto $\beta\circ\epsilon\circ\alpha$, so that $\beta\circ\epsilon\circ\alpha\in\Im(H'_\beta)(a)$. Since each element of $(\mathcal{A}\epsilon\mathcal{A})(a,b)$ is a sum of compositions $\beta\circ\epsilon\circ\alpha$ as indicated, we conclude that $(\mathcal{A}\epsilon\mathcal{A})(a,b)\subseteq\tr_{\epsilon\mathcal{A}}(\mathcal{A})(a,b)$, for all $a,b\in\Ob(\mathcal{A})$.

\smallskip
On the other hand, for the inclusion ``$\leq$'', recall that $\tr_{\epsilon\mathcal{A}}(\mathcal{A})(-,b)=\tr_{\epsilon\mathcal{A}}(H_b)$ so it is enough to prove that given a morphism $f\colon \epsilon\mathcal{A}\to  H_b$, then $\Im(f)\leq (\mathcal{A}\epsilon\mathcal{A})(-,b)$.
 Now $\epsilon\mathcal{A}$ is the image of the idempotent endomorphism $H_\epsilon\colon H_x\to  H_x$, so that $\Im(f)=\Im(g\circ H_\epsilon)$, for some morphism $g:H_x\to  H_b$. By the Yoneda Lemma, we know that $g=H_\gamma$, for some $\gamma\in\mathcal{A}(x,b)$, and $\Im(f)=\Im(H_\gamma\circ H_\epsilon)=\Im(H_{\gamma\circ\epsilon})$. Then, for any $a\in\Ob(\mathcal{A})$, 
\[
\Im(f)(a)=\{\gamma\circ\epsilon\circ\beta:\beta\in\mathcal{A}(a,x)\}     \leq\mathcal{A}(a,b)
\]
Hence, $\Im(f)(a)\leq (\mathcal{A}\epsilon\mathcal{A})(a,b)$, and so $\Im(f)\leq  (\mathcal{A}\epsilon\mathcal{A})(-,b)$, as desired. 
\end{proof}

A particular case of the above proposition is when $\E$ is a set of identities of objects in $\A$, that is, when there exists a set $\X\subseteq \Ob(\A)$ and $\E=\{\id_x:x\in\X\}$; in this case we let $\A\X\A:=\A\E\A$. A general element in $\A\X\A$ is of the form:
\[
\xymatrix@R=-2.5pt@C=18pt{
&&&&x_1\ar[ddrrrr]^{\delta_1}\\
&&&&+\\
a\ar[uurrrr]^{\gamma_1}\ar[ddrrrr]_{\gamma_k}&&&&\vdots&&&&b\\
&&&&+\\
&&&&x_k\ar[uurrrr]_{\delta_k}
}
\]
where $x_1,\dots,x_k\in \X$.

%\medskip
%Consider a set of finitely generated projectives $\P$ and let $\tr_{\P}(\A)$ be the associated trace ideal. Then, a general element $f\colon a\to b\in \A(a,b)$ is in $\tr_{\P}(\A)(a,b)$ if and only if 
%\[
%f\in \sum_{P\in \P}\sum_{\alpha\colon H_a\to P}\Im(-\circ\alpha\colon (\mod \A)(P,H_b)\to (\mod \A)(H_a,H_b))
%\]
%That is, there exist $P_1,\dots,P_n\in \P$, $\alpha_i\colon H_a\to P_i$ and $\beta_i\colon P_i\to H_b$, such that $f=\sum_{i=1}^n\beta_i\circ \alpha_i$. Now, $P$ is a summand of $H_{x_P}$ induced by an idempotent $\epsilon_P$. Notice that, for each $i=1,\dots,n$, $\beta_i\circ \alpha_i=\alpha_i\iota_i\pi_i\beta_i=(\alpha_i\iota_i)\epsilon_i(\pi_i\beta_i)\in \A\M\A$, where $\M:=\{\epsilon_P:P\in \P\}$.

\begin{proposition}\label{prop_char_ideals_trace_fgp}
Let $\A$ be a small preadditive category and consider the following conditions for an ideal $\I$ of $\A$:
\begin{enumerate}[\rm (1)]
\item $\I=\tr_{\P}(\A)$ is the trace of a class of finitely generated projective modules $\P$;
\item $\I=\A\E\A$ is generated by a set $\E$ of idempotent endomorphisms of objects of $\A$;
\item $\I=\A \X \A$, where $\X$ is a subset of $\Ob(\A)$.
\end{enumerate}
Then one always has the implications (3)$\Rightarrow$(2)$\Rightarrow$(1), while the implication (1)$\Rightarrow$(2) holds when $\A$ is additive, and the implication (2)$\Rightarrow$(3) holds when $\A$ is idempotent complete. In particular, the three conditions are all equivalent whenever $\A$ is a Cauchy complete additive category.
\end{proposition}
\begin{proof}
The implication ``(3)$\Rightarrow$(2)'' is trivial, while ``(2)$\Rightarrow$(1)" follows by Prop.\,\ref{prop.ideals generated by idempotents}. \\Suppose now that $\A$ is additive and that $\I=\tr_{\P}(\A)$ is the trace of a class of finitely presented projective modules $\P$. By Lem.\,\ref{description_fpp_lemma}, for each $P\in \P$ there is $x_P\in \Ob(\A)$ and an  idempotent endomorphism $\epsilon_P\colon x_P\to x_P$ such that $P\cong \epsilon_P\A$. Letting $\E:=\{\epsilon_P:P\in \P\}$, we obtain by Prop.\,\ref{prop.ideals generated by idempotents}
that $\I=\tr_{\P}(\A)=\A\E\A$, so the implication ``(1)$\Rightarrow$(2)" follows when $\A$ is additive. \\
Finally, suppose that $\A$ is idempotent complete and that $\I=\A\E\A$ for a set $\E$ of idempotent endomorphisms of objects of $\A$. Then, given $\epsilon\colon x_\epsilon\to x_\epsilon$ in $\E$, there are $y_\epsilon\in\Ob(\A)$ and maps $\iota_\epsilon\colon y_\epsilon\to x_\epsilon$ and $\pi_\epsilon\colon x_\epsilon\to y_\epsilon$ such that $\epsilon=\iota_\epsilon\circ\pi_\epsilon$ and $\id_{y_\epsilon}=\pi_\epsilon\circ\iota_\epsilon$. It is not difficult to verify that $\A\E\A=\A\{{y_\epsilon}:\epsilon\in\E\}\A$, so that also the implication ``(2)$\Rightarrow$(3)'' holds when $\A$ is idempotent complete.
%
%\smallskip\noindent
%(3)$\Rightarrow$(1). Let $\P:=\{H_x:\id_x\in \F\}$, and let us show that $\A\F\A=\tr_{\P}(\A)$. Indeed, a morphism $f\colon a\to b$ belongs in $\A\F\A$ if and only if it is of the form $f=(\beta_1\circ\alpha_1)+\ldots+(\beta_k\circ\alpha_k)$, with $\alpha_i\colon a\to x_i$, $\beta_i\colon x_i\to b$ and $x_i\in \F$, for $i=1,\dots,k$. Now, for each $i$, we have that $\beta_i\circ\alpha_i=(\beta_i\circ-)_{a_i}(\alpha_i)\in \Im((\beta_i\circ-)_{a_i})\subseteq H_{b_i}(a_i)$, where $\alpha_i\in (\id_{x_i}\A)(a_i)$ and $(\beta_i\circ-)\colon \id_{x_i}\A\to H_{b_i}$. Hence, each $\beta_i\circ\alpha_i$ belongs in $\tr_{\P}(\A)$. 
%
%On the other hand, given $x\in \tr_{\P}(\A)$
\end{proof}

The above proposition, together with Rem.\,\ref{rem_morita_cauchy}, allows us to find a characterization for the idempotent ideals which are traces of finitely generated projective modules:

\begin{corollary}\label{coro_ideals_vs_cauchy_with_generator}
Let $\A$ be a small preadditive category, then there is a bijection:
{\small 
\[
\xymatrix@R=0pt@C=20pt{
{\left\{\begin{matrix}\text{Full subcategories of $\proj(\A)(\cong\widehat A_{\oplus})$}\\ \text{closed under coproducts and summands}\end{matrix}\right\}}\ar@{<->}[rr]^(.55){1:1}&&{\left\{\begin{matrix}\text{Idempotent ideals of $\A$ which}\\ \text{are traces of f.g.\ projectives}\end{matrix}\right\}}\\
\X\ar@{|->}[rr]&&\tr_\X(\A)
}
\]
}
Furthermore, the above bijection restricts to the following one:
{\small 
\[
\xymatrix@R=0pt@C=20pt{
{\left\{\begin{matrix}\text{Full subcategories of $\widehat A_{\oplus}$ closed}\\ \text{under coproducts and summands}\\
\text{and with a $\oplus$-generator}\end{matrix}\right\}}\ar@{<->}[rr]^(.52){1:1}&&{\left\{\begin{matrix}\text{Idempotent ideals of $\A$ which are}\\ \text{traces of a single f.g.\ projective}\end{matrix}\right\}}
}
\]
}
\end{corollary}
\begin{proof}
Given $\X,\, \Y\subseteq \proj(\A)$, recall that $\add(\X)$ is the full subcategory spanned by the summands of finite coproducts of objects in $\X$.
Now, if $\add(\X)=\add(\Y)$, it follows by the description of the trace given in Lem.\,\ref{lem.trace-of-projectives is idempotent} that $\tr_\X(\A)=\tr_\Y(\A)$. On the other hand, let us  show that $\tr_\X(\A)=\tr_\Y(\A)$ implies $\add(\X)=\add(\Y)$. 
Indeed, suppose  $\tr_\X(\A)=\tr_\Y(\A)$, and let $P\in\mathcal{X}$. Since $P$ is finitely generated projective, there is a finite family $\{b_1,\dots,b_k\}$ in $\Ob(\mathcal{A})$ together with morphisms $u\colon P \rightarrow \coprod_{i=1}^{k} H_{b_i}$ and $p\colon \coprod_{i=1}^{k} H_{b_i} \rightarrow P$ such that $p \circ u=\id_{P}$. This implies that 
\[
\tr_{\mathcal{X}}(p)\colon \tr_{\mathcal{X}}\left(\coprod_{i=1}^{k} H_{b_i}\right)  \rightarrow \tr_{\mathcal{X}}(P)=P
\] 
is a split epimorphism. But, $\tr_{\mathcal{X}}\left(\coprod_{i=1}^{k} H_{b_i}\right)\cong\coprod_{i=1}^{k} \tr_{\mathcal{X}}(H_{b_i})$ and by assumption $\tr_{\mathcal{X}}( H_{b_i})=\tr_{\mathcal{Y}}( H_{b_i})$ for all $i=1,\dots,k$. Hence, $P \in \Gen(\mathcal{Y})$. Using once again the fact that $P$ is finitely generated projective, we deduce that $P \in \add(\mathcal{Y})$. This show that $\mathcal{X}\subseteq \add(\mathcal{Y})$, and so also $\add(\mathcal{X}) \subseteq \add(\mathcal{Y})$. By simmetry, we can then conclude that  $\add(\X)=\add(\Y)$.

Consider now the correspondence in the statement:  the assignment $\X\mapsto \tr_\X(\A)$ is surjective because, for any set of finitely generated projectives $\X'$, we have verified that $\tr_{\X'}(\A)=\tr_{\add(\X')}(\A)$, so it is enough to consider traces of families which are closed under finite coproducts and summands. Furthermore, the assignment is also injective by the first part of the proof.
\end{proof}

\subsection{Direct decompositions and central idempotents} \label{subs_directsummands}

An ideal $\I$ of a small preadditive category $\A$ is said to be a {\bf direct summand} of $\A$ when there is another ideal $\I'$ of $\A$ such that, as bi-functors $\A^{\op}\times\A\to\Ab$, we have a decomposition $\A(-,-)\cong\I(-,-)\oplus\I'(-,-)$.  When there is no risk of confusion, we  just write $\A=\I\oplus \I'$ and we call this a {\bf decomposition of $\A$ as a direct sum of ideals}.

\begin{lemma} \label{lemma_direct_summands}
Let $\A$ be a preadditive category that admits two direct sum decompositions, $\A=\I\oplus\I'$ and $\A=\I\oplus\I''$, as a direct sum of ideals. The following assertions hold:
\begin{enumerate}[\rm (1)]
\item $\I'=\I''$;
\item$ \I\cdot\I'=0=\I'\cdot\I$;
\item $\I$ is an idempotent ideal;
\item we have a decomposition $M=M\I\oplus M\I''$ in $\mod\A$, for all $\A$-modules $M$. 
\end{enumerate}
\end{lemma}
\begin{proof}
(1) Let $a,b\in\Ob(\A)$ and $\alpha\in\I'(a,b)$. Using the decomposition $\A(b,b)=\I(b,b)\oplus\I''(b,b)$, we have that $1_b=e_b+e''_b$, where $e_b\in\I(b,b)$ and $e''_b\in\I''(b,b)$. We then get that $\alpha =1_b\circ\alpha =e_b\circ\alpha +e''_b\circ\alpha$. Furthermore, $e_b\circ\alpha\in\I(a,b)\cap\I'(a,b)=0$ since $\I$ and $\I'$ are both ideals. It then follows that $\alpha =e''_b\circ\alpha\in\I''(a,b)$ since $\I''$ is an ideal of $\A$. Therefore, $\I'\subseteq\I''$ and, by  symmetry, the converse inclusion also holds. 

\smallskip\noindent
(2) This  is trivial since $ \I\cdot\I'\subseteq\I\cap\I'\supseteq\I'\cdot\I$, and $\I\cap\I'=0$. 

\smallskip\noindent
(3) Taking  $\alpha\in\I(a,b)$ and arguing as in the proof of assertion (1), we get a decomposition $\alpha =e_b\circ\alpha +e''_b\circ\alpha$. But the second summand is in $\I\circ\I''=0$. Then $\alpha\in\I^2(a,b)$ since $e_b\in\I(b,b)$. 

\smallskip\noindent
(4) If we put $N:=M\I\cap M\I'$, then  $N\I=0=N\I'$ using assertion (2). On the other hand, viewing $\A$ as an ideal of itself in the obvious way, we have that $M=M\A=M(\I+\I'')\subseteq M\I+M\I''$. It then follows that $M=M\I\oplus M\I''$. 
\end{proof}

As in the case of unital rings, we have the following result.

\begin{proposition} \label{prop.bijection-central-idempotents}
Let $\A$ be a small preadditive category and let 
\[
\S_1:=\{\text{idempotent elements in $Z(\A )$}\}\quad\text{and}\quad
\S_2:=\{\text{direct summands  of $\A$}\}.
\]
The assignment $\epsilon\mapsto\I_\epsilon$, where $\I_\epsilon$
 is the ideal of $\A$ defined as 
\[
\I_\epsilon (a,b):=\{u\in\A(a,b):\epsilon_b\circ u=u\text{ (equivalently, }u\circ\epsilon_a=u\text{)}\},
\]
for all $a,b\in\Ob(\A)$, defines a bijection $\Phi\colon \S_1\overset{\cong }{\to }\S_2$. 
\end{proposition}

Let us remark that, given an idempotent $\epsilon\in Z(\mathcal{A})$, then $\mathcal{I}_\epsilon =\mathcal{A}\mathcal{E}_\epsilon\mathcal{A}$, where\linebreak $\mathcal{E}_\epsilon:=\{\epsilon_a : a\in\Ob(\mathcal{A})\}$.

\begin{proof}
Let $u\in\I_\epsilon (a,b)$ and let us consider morphisms $r\colon a'\to a$ and $l\colon b\to b'$ in $\A$. Then, using that $\epsilon\colon \id_\A\to \id_\A$ is a natural transformation and that $\epsilon_b\circ u=u$, we get the following equalities 
\[
\epsilon_{b'}\circ l\circ u\circ r=l\circ\epsilon_b\circ u\circ r= l\circ u\circ r,
\] 
proving that $l\circ u\circ r\in\I_\epsilon (a',b')$, so that $\I_\epsilon$ is an ideal of $\A$. Note that we have not used yet the idempotency of $\epsilon$. That is, for each $\epsilon\in Z(\A)$, we have a well-defined ideal $\I_\epsilon$ of $\A$. 

We now consider, for all $a,\, b\in\Ob(\A)$, the subgroup $\I'_\epsilon (a,b)$ of $\A(a,b)$ consisting of the morphisms $v\colon a\rightarrow b$ such that $\epsilon_b\circ v=0$ (or, equivalently, $v\circ\epsilon_a=0$).  Using again  that $\epsilon$ is a natural transformation, one readily sees that the $\I'_\epsilon (a,b)$'s define an ideal $\I'_\epsilon$ of $\A$ such that $\I_\epsilon(a,b)\cap\I'_\epsilon (a,b)=0$, for all $a,\, b\in\Ob(\A)$.  Moreover, for each $w\in\A(a,b)$, we have a decomposition $w=\epsilon_b\circ w +(w-\epsilon_b\circ w)$, where $\epsilon_b\circ w\in\I_\epsilon (a,b)$ and $w-\epsilon_b\circ w\in\I'_\epsilon (a,b)$ due to the idempotency of $\epsilon$. Therefore we have a decomposition as a direct sum of ideals $\A =\I_\epsilon\oplus\I'_\epsilon$, so that the assignment $\epsilon\mapsto\I_\epsilon$ gives a well-defined map $\Phi \colon\S_1\to\S_2$. 

To check the injectivity of $\Phi$, consider two idempotents $\epsilon ,\, \epsilon'\in Z(\A)$ and suppose that $\I_\epsilon =\I_{\epsilon'}$, i.e. $\Phi(\epsilon )=\Phi (\epsilon')$. We then have that $\epsilon_a,\,\epsilon'_a\in\I_\epsilon (a,a)=\I_{\epsilon'}(a,a)$, for all $a\in\Ob(\A)$. In particular, $\epsilon'_a\circ\epsilon_a=\epsilon_a$ and, using the naturality of $\epsilon$, we also have that $\epsilon'_a\circ\epsilon_a=\epsilon_a\circ\epsilon'_a=\epsilon'_a$ since $\epsilon'_a\in\I_\epsilon (a,a)$. Therefore $\epsilon_a=\epsilon'_a$, for all $a\in\Ob(\A)$, which implies that $\epsilon =\epsilon'$.

For the surjectivity of $\Phi$, let us take a direct summand $\I$ of $\A$ and fix a decomposition $\A=\I\oplus\I'$, which is unique according to Lemma \ref{lemma_direct_summands}. We then have a decomposition $\A(a,a)=\I(a,a)\oplus\I'(a,a)$ in $\Ab$, for all $a\in\Ob(\A)$. This gives a decomposition $1_a=\epsilon_a+\epsilon'_a$, with $\epsilon_a\in\I(a,a)$ and  $\epsilon'_a\in\I'(a,a)$, for all $a\in\Ob(\A)$. Note that, for each $u\in\I(a,b)$, we have $u=1_b\circ u=\epsilon_b\circ u+\epsilon'_b\circ u$, where the second summand belongs in  $(\I'\cdot\I)(a,b)=0$. Then $u=\epsilon_b\circ u$. 

It remains to check that the collection $\epsilon :=(\epsilon_a:a\rightarrow a)$ defines a natural transformation $\epsilon\colon\id_\A\to \id_\A$, which will be clearly idempotent and will satisfy that $\Phi (\epsilon )=\I$.  Indeed let $u\colon a\to b$ be a morphism in $\A$. Bearing in mind that $\I$ and $\I'$ are ideals, the decompositions $u=u\circ\epsilon_a+u\circ\epsilon'_a$ and $u=\epsilon_b\circ u+\epsilon'_b\circ u$ are both the decomposition of $u$ with respect to the decomposition $\A(a,b)=\I(a,b)\oplus\I'(a,b)$. By uniqueness, we then get $\epsilon_b\circ u=u\circ\epsilon_a$, and hence $\epsilon$ is a natural transformation. 
\end{proof}

\begin{corollary} \label{cor.description-TTF}
Let $\A$ be a small preadditive category, let $\epsilon=\epsilon^2\in Z(\A)$ be any idempotent element, let $\widehat{\epsilon}$ the element of $Z(\mod\A)$ corresponding to $\epsilon$ by the bijection of Proposition \ref{prop.Z(A)-vs-Z(Mod-A)} and let $\I_\epsilon$ the direct summand of $\A$ corresponding to $\epsilon$ by the bijection of Proposition \ref{prop.bijection-central-idempotents}. For a right $\A$-module $M$, the following assertions hold:
\begin{enumerate}[\rm (1)]
\item $M=M\I_\epsilon$ if, and only if,   $\widehat{\epsilon}_M$ is an isomorphism if, and only if, $\widehat{\epsilon}_M=\id_M$;
\item $M\I_\epsilon =0$ if, and only if, $\widehat{\epsilon}_M=0$. 
\end{enumerate}
\end{corollary}
\begin{proof}
 Since we have that $\widehat{\epsilon}_M\circ\widehat{\epsilon}_M=\widehat{\epsilon}_M$, we immediately get that $\widehat{\epsilon}_M$ is an isomorphism if, and only if, $\widehat{\epsilon}_M=\id_M$. 

By definition we have that $(M\I_\epsilon)(a)=\sum_{\alpha\in\I_\epsilon (a,b)}\Im(M(\alpha))$. By the proof of the last proposition, we know that $\alpha\in\I_\epsilon (a,b)$ if and only if $\epsilon_b\circ\alpha =\alpha$, which is equivalent to say that $\alpha =\alpha\circ\epsilon_a$ due to the naturality of $\epsilon$. It then follows that $M(\alpha )=M(\epsilon_a)\circ M(\alpha)$, for all morphisms $\alpha$ in $\I$ with domain $a$. It easily follows   that $(M\I_\epsilon)(a)=\Im(M(\epsilon_a))$. 
But, by the proof of  Proposition \ref{prop.Z(A)-vs-Z(Mod-A)}, we know that the evaluation of $\widehat{\epsilon}_M$ at $a$ is $\widehat{\epsilon}_{M,a}=M(\epsilon_a)$. It then follows that $(M\I_\epsilon)(a)=\Im(\widehat{\epsilon}_{M,a})$. Hence, $M\I_\epsilon =0$ if, and only if, $\widehat{\epsilon}_{M,a}=0$, for all $a\in\Ob(\A)$, verifying assertion (2).

  On the other hand, we also have that $M=M\I_\epsilon$ if, and only if,  the induced map $\widehat{\epsilon}_{M,a}=M(a)\to M(a)$ is an epimorphism, for all $a\in\Ob(\A)$. But this latter map is an idempotent endomorphism of $M(a)$. Then it is an epimorphism if and only if it is an isomorphism.  Hence assertion (1) also follows easily. 
\end{proof}

\begin{remark} \label{rem.clarification}
Although, in order to avoid redundancy, we have not said it explicitly in the statement of the above corollary, it is clear from its proof that $M\I_\epsilon=M$ if, and only if, $M(\epsilon_a)$ is an isomorphism (equivalently $M(\epsilon_a)=\id_{M(a)}$), for all $a\in\Ob(\A)$. It is also clear that $M\I_\epsilon =0$ if, and only if,  $M(\epsilon_a)=0$ for all $a\in\Ob(\A)$. 
\end{remark}

\section{\hbox{Hereditary torsion vs  Grothendieck topologies}}

The section is organized as follows: we start recalling some basic facts and definitions about torsion pairs in Sec.\ref{recall_torsion_subs}; we  prove our generalization of Gabriel's bijection in Sec.\ref{grabriel_subs}, specializing  this result to hereditary torsion pairs of finite type in Sec.\ref{subs_torsion_fin_type}. %We extend Jan's correspondence in Sec.\ref{jans_subs} and, in Sec.\ref{subs_centrally}, we specialize this correspondence to splitting TTF triples. 

\subsection{Torsion pairs}\label{recall_torsion_subs}

Recall the concept of a torsion pair from the Introduction, that we only consider in the category $\mod\A$ in the rest of the paper. 
 If $X$ is an $\A$-module and we consider the torsion  sequence 
 \[
 0\rightarrow T_X\to X\to F_X\to 0,
 \] 
 then $T_X$ and $F_X$ depend functorially on $X$, and the corresponding functors $t\colon \mod \A\to \T$ and $(1:t)\colon \mod \A\to \F$
are called, respectively, the {\bf torsion radical} and the {\bf torsion coradical} functors. In fact, $t$ is the right adjoint of the inclusion $\T\to \mod \A$, while $(1:t)$ is the left adjoint to the inclusion $\F\to \mod\A$. We can visualize this situation as in the following diagram:
\[
\xymatrix{
\T\ar[rr]|{\text{inclusion}}&&\mod \A\ar@/_-12pt/[ll]|t\ar[rr]|{(1:t)}&&\F\ar@/_-12pt/[ll]|{\text{inclusion}}
}
\]
A torsion pair ${\t}=(\T,\F)$ is said to be {\bf hereditary} provided $\T$ is closed under taking subobjects.

\medskip
The following lemma is well-known (see, for example, \cite{S}):

\begin{lemma}\label{construction_of_tor_rad}
A class $\T\subseteq \mod \A$ is a torsion class (resp., a torsionfree class) if and only if it is closed under taking quotients, extensions and coproducts (resp., subobjects, extensions and products).
\end{lemma}
%\begin{proof}
%It is easy to prove that any torsion class is closed under taking quotients, extensions and coproducts. On the other hand, given $\T$ with these closure properties and letting $\F:=\T^{\perp}$, it is not difficult to verify that $\T={}^{\perp}\F$ so (Tors.1) is verified. Furthermore, given $X\in \Ob(\G)$, let $T_X:=\sum\{T\leq X:T\in \T\}$ and $F_X:=X/T_X$. It is then easy to prove that $T_X\in \T$ and $X/T_X\in \F$, so also (Tors.2) is satisfied.
%\end{proof}

Any torsion pair induces an ideal of $\A$, as shown in the following lemma:

\begin{lemma}\label{torsion_ideal_lem}
Let $\mathbf{t}=(\T,\F)$ be a torsion pair in $\mod{\A}$. There is an ideal $t(\A)$ of $\A$ defined as $t(\A)(a,b):=t(H_b)(a)$, for all $a,b\in\Ob(\A)$.
%
%We can form a new preadditive category $\A/t(\A)$ as follows. We put $\Ob(\A/t(\A))=\Ob(\A)$ and, given any objects $x,y\in\Ob(\A)$, we put $(\A/t(\A))(x,y):=[(1:t)(\A(-,y))](x)$, where $t:\mod{\A}\to\mod{\A}$ is the torsion radical associated to $\mathbf{t}$. 
\end{lemma}
\begin{proof}
Let $f\in t(\A)(a,b)$, $r\in \A(a',a)$, $l\in \A(b,b')$,  and let us verify that $f\circ r\in t(\A)(a',b)$ and that $l\circ f\in t(\A)(a,b')$. Indeed, since the torsion radical  $t$ is an additive subfunctor of $\id_{\mod{\A}}$, then $\A(-,l)(t(H_b))\leq t(H_{b'})$. In particular, $f\circ l=\A(a,l)(f)\in  t(H_{b'})(a)=t(\A)(a,b')$. On the other hand, the fact that $f\in t(\A)(a,b)$, means exactly that the image of the map $\A(-,f)\colon H_a\to H_b$ is torsion, thus we can corestrict to obtain a map $\A(-,f)\colon H_a\to t(H_b)$. It is now clear that also the composition $\A(-,f)\circ\A(-,r)\colon H_{a'}\to H_b$  takes values in $t(H_b)$, so that $f\circ r\in t(\A)(a',b)$.
\end{proof}

\subsection{Gabriel's bijection}\label{grabriel_subs}

The following definition appears in \cite{RG} and earlier, under the name of ``linear topology'', in \cite{Lo}, and it can be thought of as an additive version of the notion of Grothendieck topology (see, for example, \cite{MM}).

\begin{definition} \label{def.Grothendieck-topology}
A {\bf (linear) Grothendieck topology} on $\A$ is a family $\mathbf{G}=\{\mathbf{G}_a:a\in\Ob(\A)\}$, where $\mathbf{G}_a$ is a set of submodules of the representable $\A$-module $\A(-,a)$, for each $a\in\Ob(\A)$, satisfying:
\begin{itemize}
\item[\rm ({\bf Id})] the {\bf Identity axiom}, $H_a\in\mathbf{G}_a$, for each  $a\in\Ob(\A)$;
\item[\rm ({\bf Pb})] the {\bf Pullback axiom}, for $R\in\mathbf{G}_a$ and $r\colon a'\to a$ in $\A$, consider the following pullback square:
\[
\xymatrix{
r^{-1}R\, \ar@{^(.>}[rr]\ar@{.>}[d]\ar@{}[drr]|{\text{\scriptsize P.B.}}&&H_{a'}\ar[d]^{\A(-,r)}\\
R\, \ar@{^(->}[rr]&&H_a.
}
\] 
Then, $r^{-1}R$ is in $\mathbf{G}_{a'}$;
\item[\rm ({\bf Glue})] the {\bf Glueing axiom}, given $R\leq H_a$, suppose that there exists $S\in\mathbf{G}_a$ such that, for any $a'\in\Ob(\A)$ and any $r\in S(a')\leq\A(a',a)$, one has that $r^{-1}R$ is in $\mathbf{G}_{a'}$. Then, $R\in\mathbf{G}_a$. 
\end{itemize}
\end{definition}

Let us remark that, if a ring $R$ is viewed as a preadditive category with one object,  then ``Grothendieck topology'' and ``Gabriel topology'' on $R$ are synonymous. The following result is part of \cite{Lo} and \cite[Prop.\,1.8]{RG}.

\begin{lemma} \label{lem.axioms-Up-Int}
Let $\mathbf{G}=\{\mathbf{G}_a:a\in\Ob(\A)\}$ be a Grothendieck topology on $\A$. The following assertions hold true, for each $a\in\Ob(\A)$:
\begin{enumerate}[\rm (1)]
\item if $R\leq S\leq H_a$ and $R\in\mathbf{G}_a$, then also $S\in\mathbf{G}_a$;
\item if $R_1,\dots ,R_m\in\mathbf{G}_a$, then $R_1\cap\dots \cap R_m\in\mathbf{G}_a$.
\end{enumerate}
\end{lemma}

We are now going to prove that Grothendieck topologies on $\A$ are in bijection with hereditary torsion pairs in $\mod \A$ (see Thm.\,\ref{thm.Gabriel bijection for small categories}). In the following lemma we show how a Grothendieck topology induces a hereditary torsion class, while the opposite direction is illustrated in Lem.\,\ref{defining_psi}.

\begin{lemma}\label{defining_phi}
Let $\mathbf{G}=\{\mathbf G_a:a\in\Ob(\A)\}$ be a Grothendieck topology on $\A$ and define 
\[
\mathcal{T}^\mathbf{G}:=\Gen\{H_a/R : a\in\Ob(\A),\, R\in\mathbf{G}_a\}\subseteq \mod{\A}.
\] 
Then the following statements hold true:
\begin{enumerate}[\rm (1)]
\item an $\A$-module $T$ is in $\mathcal{T}^\mathbf{G}$ if, and only if, $\Ker(\varphi)\in \mathbf{G}_a$ for every map $\varphi \colon H_a\to T$;
\item $\T^\mathbf{G}$ is a hereditary torsion class in $\mod{\A}$.
\end{enumerate} 
\end{lemma}
\begin{proof}
(1) Given $T\in\mathcal{T}^\mathbf{G}$, there is an epimorphism $p\colon\coprod_{i\in I}(H_{a_i}/R_i)\twoheadrightarrow T$, where $R_i\in\mathbf{G}_{a_i}$, for each $i\in I$.  Furthermore, given $\varphi\colon H_a\to T$, one can use that $H_a$ is finitely generated projective to show that there is a morphism $\bar\varphi\colon H_a\to\coprod_{i\in F}(H_{a_i}/R_i)$ for a finite subset $F\subseteq I$, such that $\varphi=p\circ\iota_F\circ \bar\varphi$, where $\iota_F\colon \coprod_{i\in F}(H_{a_i}/R_i)\to \coprod_{i\in I}(H_{a_i}/R_i)$ is the inclusion. Clearly, $\bar\varphi$ is described by a vector $\bar\varphi=(\bar\varphi_i)_{i\in F}$ with $\bar\varphi_j\colon H_a \to (H_{a_i}/R_j)$, for all $j\in F$ and, using again the projectivity of $H_a$ and the Yoneda Lemma, each of the morphisms $\bar\varphi_j$ factors in the form $\bar\varphi_j=p_j\circ \A(-,r_j)$, where $p_j\colon H_{a_j}\to H_{a_j}/R_j$ is the projection and $r_j \colon a\to a_j$ is a suitable morphism in $\A$. To conclude, notice that $\Ker(\bar\varphi)=\bigcap_{i\in F}\Ker(\bar\varphi_i)$ so, by Lem.\,\ref{lem.axioms-Up-Int}, it is enough to verify that $\Ker(\bar\varphi_j)\in\mathbf{G}_{a_j}$, for all $j\in I$. But, with the notation of Def.\,\ref{def.Grothendieck-topology}, $\Ker(\bar\varphi_j)=r_j^{-1}R_j$, and so $\Ker(\bar\varphi_j)\in\mathbf{G}_a$ by the axiom ({\bf Pb}).

On the other hand, given $T\in\mod \A$  such that $\Ker(\varphi)\in \mathbf{G}_a$ for every morphism \mbox{$\varphi \colon H_a\to T$}, consider  an epimorphism $q\colon\coprod_{i\in I}H_{a_i}\twoheadrightarrow T$ and take the compositions $q\circ\iota_j\colon H_{a_j}{\to}T$, where $\iota_j\colon H_{a_j}\to \coprod_{i\in I}H_{a_i}$ is the inclusion, so that $R_j:=\Ker (q\circ\iota_j)\in\mathbf{G}_{a_j}$, for all $j\in I$. We then get an induced epimorphism $\coprod_{i\in I}(H_{a_i}/R_i)\twoheadrightarrow T$, showing that $T\in\T^\mathbf{G}$.

\smallskip\noindent
(2) follows by part (1) and \cite[Prop.\,2.7]{Lo}.
\end{proof}

\begin{lemma}\label{defining_psi}
Let $\T$ be a hereditary torsion class in $\mod{\A}$ and define 
\[
\mathbf{G}^\mathcal{T}:=\{\mathbf{G}^\T_a:a\in\Ob(\A)\},\quad\text{with}\quad \mathbf{G}^\T_a:=\{R\leq\A(-,a):\A(-,a)/R\in\T\}.
\] 
Then, $\mathbf{G}^\mathcal{T}$ is a Grothendieck topology on $\A$. 
\end{lemma}
\begin{proof}
It is immediate to check that $\mathbf{G}$ satisfies axioms ({\bf Id}) and  ({\bf Pb}) of Def.\,\ref{def.Grothendieck-topology}, so we only need to check axiom  ({\bf Glue}).  Consider a submodule $R\leq H_a$  and suppose that there exists $S\in\mathbf{G}^\T_a$ such that, for any $a'\in\Ob(\A)$ and $r\in S(a')\leq\A(a',a)$,  $r^{-1}R$ is in $\mathbf{G}^\T_{a'}$. Consider the following short exact sequence:
\[
0\rightarrow S/(S\cap R)\to H_a/R\to H_a/(S+R)\rightarrow 0.
\] 
Then $H_a/(S+R)\in\T$ since it is a quotient of $H_a/S$. Hence, it is enough to verify that ${S}/({S\cap R})\in\T$. Indeed, fix an epimorphism $q\colon \coprod_{i\in I}H_{a_i}\twoheadrightarrow S$ and, for each $j\in I$, consider the following composition of $q$ with the obvious inclusions:
\[
\xymatrix{
\varphi_j\colon H_{a_j}\ar@{->}[r]^{}&\coprod_{i\in I}H_{a_i}\ar@{->>}[r]^{q}& S\leq H_{a}.
}
\] 
By hypothesis, $\varphi_j^{-1}R\in \mathbf{G}^\T_{a_j}$ for any $j\in I$, so that $H_{a_j}/\varphi_j^{-1}R\in \T$. Furthermore, there is clearly an epimorphism $\coprod_{i\in I}(H_{a_i}/\varphi_i^{-1}R)\twoheadrightarrow S/(S\cap R)$, so that $S/(S\cap R)\in \T$ as desired.
%
%and there is clearly an epimor
%
%
%
%Then, for each $j\in I$, we obtain the following commutative diagram, composition of pullbacks:
%\[
%\xymatrix{\varphi_j^{-1}(R) \ar[r] \ar[d] \pushoutcorner & S \cap R \ar@{^(->}[r] \pushoutcorner \ar@{^(->}[d] & R \ar@{^(->}[d] \\ \A(-,a_j) \ar[r]_{q_j} & S \ar@{^(->}[r] & \A(-,a).}
%\]
%It follows that we have an induced monomorphism $\tilde{q}_j\colon \frac{\A(-,a_j)}{\varphi_j^{-1}(R)}\rightarrowtail\frac{S}{S\cap R}$, for each index $j\in I$. We then get an induced morphism $\bar{q}\colon \coprod_{i\in I}\frac{\A(-,a_i)}{\varphi_i^{-1}(R)}\to\frac{S}{S\cap R}$, which is necessarily an epimorphism since so is $q\colon \coprod_{i\in I}\A(-,a_i)\twoheadrightarrow S$. We then get that $\frac{S}{S\cap R}\in\T$, as desired. 
\end{proof}

\begin{theorem} \label{thm.Gabriel bijection for small categories}
Let $\A$ be a small preadditive category. Then there is a one-to-one correspondence
\[
\xymatrix@R=0pt{
\Phi:\S_1:=
{\left\{
\begin{matrix}
\text{Grothendieck}\\ 
\text{topologies on $\A$}
\end{matrix}
\right\}}
\ar@{<->}[rr]^{1:1}&&
{\left\{
\begin{matrix}
\text{Hereditary torsion}\\ 
\text{classes in $\mod{\A}$}
\end{matrix}
\right\}}
=:\S_2:\Psi\\
\mathbf{G}\ar@{|->}[rr]&&\Phi(\mathbf{G}):=\mathcal{T}^\mathbf{G}\\
\Psi(\T):=\mathbf{G}^\mathcal{T}\ar@{<-|}[rr]&&\T
}
\]
where $\T^{\mathbf G}$ and $\mathbf G^{\T}$ are defined as in Lem.\,\ref{defining_phi} and \ref{defining_psi}, respectively.
%
%, for each $a\in\Ob(\A)$, $\mathbf{G}^\T_a:=\{R\leq\A(-,a):\A(-,a)/R\in\T\}$. 
\end{theorem}
\begin{proof}
The maps  $\Phi$ and $\Psi$ are well-defined by Lem.\,\ref{defining_phi} and \ref{defining_psi}; let us verify that they are inverse bijections. Consider first $\T\in \S_2$ and let us verify that $\T$ is equal to $\Phi\circ\Psi (\T)=\Gen(H_a/R:R\in \mathbf G_a^\T,\, a\in \Ob(\A))$. In fact, by the very definition of $\mathbf{G}_a^\T$, given $R\in\mathbf{G}^{\T}_a$, $H_a/R\in\T$ so $\Phi\circ\Psi(\T)\subseteq \T$. On the other hand, given $T\in \T$, consider an epimorphism $\coprod_{i\in I}H_{a_i}\twoheadrightarrow T$, that induces an epimorphism $\coprod_{i\in I}(H_{a_i}/R_i)\twoheadrightarrow T$, where each $H_{a_j}/R_j$ is a subobject of $T$, so it belongs to $\T$, that is, $R_j\in G_{a_j}^\T$. This shows that $T\in \Phi\circ\Psi(\T)$. 

On the other hand, given $\mathbf{G}\in\mathcal{S}_1$, let us show that $\mathbf G=\Psi\circ \Phi (\mathbf G)$, that is $\mathbf G_a=\mathbf G^{\T^{\mathbf G}}_a$, for each $a\in \Ob(\A)$. Indeed, given $R\in \mathbf G_a$, it is clear that $H_a/R\in \T^{\mathbf G}$, so that $R\in \mathbf G^{\T^{\mathbf G}}_a$ and $\mathbf G_a\subseteq\mathbf G^{\T^{\mathbf G}}_a$. On the other hand, given $S\in \mathbf G^{\T^{\mathbf G}}_a$, that is, given an $S\leq H_a$ such that $H_a/S\in \T^{\mathbf G}$, we known by Lem.\,\ref{defining_phi} that the kernel of any morphism of the form $H_b\to H_a/S$, for some $b\in\Ob(\A)$, is in $\mathbf G_b$. In particular, the kernel of the obvious projection $H_a\to H_a/S$, which is exactly $S$, does belong in $\mathbf G_a$, so that $\mathbf G^{\T^{\mathbf G}}_a\subseteq \mathbf G_a$, as desired. 
%
%by Lemma \ref{defining_phi}, $\T^\mathbf{G}$ is the class of those $T$ such that the kernel of any morphism $\A(-,a)\to T$ is in $\mathbf{G}_a$. But, by definition of $\mathbf{G}':=\Psi (\T^\mathbf{G})$, we immediately get that $\mathbf{G}_x'\subseteq\mathbf{G}_x$, for each $x\in\Ob(\A)$. But if $R\in\mathbf{G}_x$, then $\frac{\A(-,x)}{R}\in\T^\mathbf{G}$ since  $\mathcal{T}^\mathbf{G}:=\Gen(\A(-,x)/R:x\in\Ob(\A))\text{ and }R\in\mathbf{G}_x)$. It then follows that $R\in G'_x$, so that we also have $G_x\subseteq G'_x$, for each $x\in\Ob(\A)$. Therefore $\mathbf{G}':=(\Psi\circ\Phi)(\mathbf{G})=\mathbf{G}$, and so $\Phi$ and $\Psi$ inverse to each other. 
\end{proof}

\subsection{Hereditary torsion classes of finite type}\label{subs_torsion_fin_type}

Recall that a torsion pair $(\T,\F)$ in $\mod \A$ is said to be of {\bf finite type} provided $\F=\varinjlim \F$, that is, if $\F$ is closed under taking direct limits. As a well-known consequence of the definition, one can verify that both the torsion radical and the torsion coradical associated with a torsion pair of finite type do preserve direct limits. 

\begin{lemma} \label{prop.hereditary of finite type}
Let ${\t}=(\T,\F)$ be a hereditary torsion pair in $\mod \A$. Then ${\t}$ is of finite type if and only if there is a set $\mathcal{S}\subseteq \fpmod \A$ (where $\fpmod \A$ is the category of finitely presented $\A$-modules) such that $\F=\S^{\perp}$.
\end{lemma}
\begin{proof}
We only need to check the ``only if'' part, which is an easy adaptation of the proof for modules over a ring (see the proof of implication (2)$\Rightarrow$(1) in \cite[Prop.\,42.9]{Go}), which we just outline. We need to check that each $T\in\T\cap\fg(\mod \A)$ (i.e., any finitely generated torsion $\A$-module) is a quotient of an object in $\T\cap\fpmod \A$. Consider an exact sequence 
\begin{equation}\label{torsion_fg_torsion_fp}
\xymatrix@C=15pt{
0\ar[r]& R\ar[rr]^{\lambda}&&X\ar[rr]^p&&T\ar[r]&0,
}
\end{equation}
with $T\in\T\cap\fg(\mod \A)$ and $X\in \fpmod \A$. Express $R$ as a direct union of its finitely generated subobjects $R=\bigcup_{j\in J}R_j$ and, for any $j\in J$, consider the following diagrams:
\[
\xymatrix@C=9pt@R=2pt{
&&&&&&&&0\ar[r]& R_j\ar@{=}[dd]\ar[rr]^{\lambda_j'}&&X_j\ar[dd]\ar[rr]^{p_j'}&&t(X/R_j)\ar[dd]\ar[r]&0\\
0\ar[r]& R_j\ar[rr]^{\lambda_j}&&X\ar[rr]^{p_j}&&X/R_j\ar[r]&0\\
&&&&&&&&0\ar[r]& R_j\ar[rr]^{\lambda_j}&&X\ar[rr]^{p_j}&&X/R_j\ar[r]&0
}
\]
where the one on the right hand side is obtained with a pullback from the other. Letting $j$ vary in $J$, we obtain two direct systems of short exact sequences $(0\rightarrow R_j\stackrel{}{\rightarrow}X\stackrel{}{\to}X/R_j\rightarrow 0)_{j\in J}$ and $(0\rightarrow R_j\stackrel{}{\rightarrow}X_j\stackrel{}{\to}t(X/R_j)\rightarrow 0)_{j\in J}$ whose direct limit is the sequence \eqref{torsion_fg_torsion_fp} (this is clear for the first sequence, while for the second one it is enough to use that the torsion radical preserves direct limits). Now, since $X$ is finitely presented, there is some $k\in J$ such that the canonical map $u_k\colon X_k\to X$  is a retraction, so that we obtain the following commutative diagram with exact rows:
\[
\xymatrix@C=30pt{
0 \ar[r] & R_k \ar[r]^{} \ar@{^(->}[d]& X_k \ar[r]^{} \ar@{>>}[d]^{u_k} & t(X/R_k) \ar[r] \ar[d]^{\alpha_k} & 0\\ 
0 \ar[r] & R \ar[r]_{} & X \ar[r]_{} & T \ar[r] & 0.
}
\]
Since $\T$ is hereditary, $\Ker (\alpha _k)\in \T$ and, applying the Snake Lemma to the above diagram, we obtain that $R/R_k$ is a quotient of $\Ker (\alpha _k)$, so that $R/R_k\in \T$. As a consequence, $X/R_k\in\T$, since it is an extension of $R/R_k$ and $(X/R_k)/(R/R_k)\cong X/R\in \T$. Furthermore, $X/R_k$ is also finitely presented, since $X\in \fpmod \A$ and $R_k\in \fg(\mod \A)$, hence $T\cong X/R$ is a quotient of $X/R_k\in \fpmod\A\cap \T$, as desired. 
\end{proof}

Recall that a {\bf basis} for a Grothendieck topology $\mathbf{G}=\{\mathbf{G}_a:a\in\Ob(\A)\}$ on $\A$ (see \cite{Lo}) is a family $\mathbf{B}=\{\mathbf{B}_a:a\in\Ob(\A)\}$ such that 
\begin{itemize}
\item $\mathbf{B}_a\subseteq\mathbf{G}_a$, for all $a\in\Ob(\A)$;
\item for each $R\in\mathbf{G}_a$ there exists $S\in\mathbf{B}_a$ such that $S\leq R$.
\end{itemize} 
We shall say that $\mathbf{B}$ is  {\bf a basis of finitely generated right ideals} of $\mathbf{G}$, when all the right $\A$-modules $R\in \mathbf{B}_a$ are finitely generated, for all $a\in \Ob(\A)$. As for modules over associative unital rings and, more generally, in locally finitely presented Grothendieck categories (see \cite[Prop.\,11.1.4]{Pr}), we have:

\begin{proposition} \label{prop.basis of finitely generated ideals}
A hereditary torsion pair ${\t}=(\T,\F)$ in $\mod{\A}$ is of finite type if, and only if, the Grothendieck topology $\mathbf G^{\T}$ (see Thm.\,\ref{thm.Gabriel bijection for small categories}) has a basis of finitely generated ideals. 
\end{proposition}
\begin{proof}
If  $\mathbf G^{\T}$ has a basis of finitely generated ideals, then let $\S:=\{H_a/R:a\in \Ob(\A),\, R\in \mathbf{G}_a\cap \fg(\mod \A)\}$ and notice that $\F={\S}^{\perp}$, so that ${\t}$ is of finite type. Conversely, let  ${\t}$ be of finite type and, for each $a\in\Ob(\A)$, define $\mathbf{B}_a:=\{R\leq H_a:H_a/R\in\T\cap \text{mod} \A)\}=\mathbf{G}_a^\T\cap\text{fg}(\mod \A)$. Let $R\in\mathbf{G}_a^\mathbf{\T}$ be arbitrary. The proof of Lem.\,\ref{prop.hereditary of finite type}, with $X=H_a$ and $T=H_a/R$, gives a finitely generated subobject $R_k\subseteq R$ such that $H_a/R_k\in\T$. This just says that $R_k\in\mathbf{B}_a$, so that $\mathbf{B}:=\{\mathbf{B}_a:a\in\Ob (\A)\}$ is a basis of $\mathbf{G}^\T$ of finitely generated right ideals. 
%
%
%One only needs to prove the ``only if'' part, for which we just need to consider the associated Grothendieck topology $\mathbf{G}$ and then take  in the proof of Prop.\,\ref{prop.hereditary of finite type}.
\end{proof}

\section{TTF triples, idempotent ideals and recollements}

The section is organized as follows: we start recalling some basic facts and definitions about TTF triples in Sec.\ref{sub_TTF triples and Abelian recollements}, including their relation with Abelian recollements; we then extend Jan's correspondence in Sec.\ref{jans_subs}, showing a bijection between the family of TTF classes in $\mod \A$ and idempotent ideals of $\A$. In Sec.\,\ref{subs_Abelian recollements of module categories} we specialize Jan's Theorem showing that Abelian recollements of $\mod \A$ by categories of modules correspond to those idempotent ideals that are traces of finitely generated projectives. Finally, in Sec.\ref{subs_centrally}, we show that Jan's Theorem induces a correspondence between (idempotent ideals generated by) central idempotents and splitting TTF triples.

\subsection{TTF triples and Abelian recollements}\label{sub_TTF triples and Abelian recollements}

A hereditary torsion class $\T$ is said to be a {\bf TTF class} (torsion and torsionfree class), provided it is closed under taking products. By Lem.\,\ref{construction_of_tor_rad}, both $(\T,\T^\perp=:\F)$ and $(\C:={}^{\perp}\T,\T)$ are torsion pairs, we denote by $t\colon \mod \A\to \T$ and by $c\colon \mod \A\to \C$ the corresponding torsion radicals; the triples of the form $(\C,\T,\F)$ are called {\bf TTF triples}. In this situation, we obtain a diagram as follows:
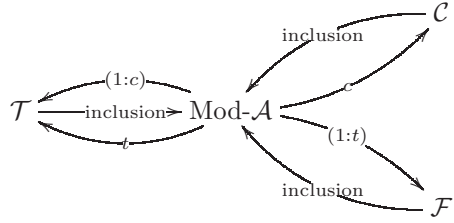
\[
\xymatrix{
&&&&\C\ar@/_12pt/[lld]|{\text{inclusion}}\\
\T\ar[rr]|{\text{inclusion}}&&\mod \A\ar@/_-12pt/[ll]|t\ar@/_12pt/[ll]|{(1:c)}\ar@/_10pt/[urr]|{c}\ar@/_-10pt/[rrd]|{(1:t)}&&\\
&&&&\F\ar@/_-12pt/[llu]|{\text{inclusion}}\\
}
\]
In fact, there is an alternative way to think about TTF triples, that is, these triples are in bijection with (equivalence classes of) the ``{Abelian recollements}'' of $\mod \A$. 
\begin{definition}
A {\bf recollement $\mathcal{R}$ of $\mod\A$ by Abelian categories} $\X$ and $\Y$ (also called an {\bf Abelian recollement}) is a diagram of additive functors
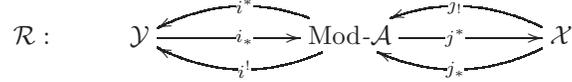
\[
\xymatrix{
\mathcal{R}:&\Y\ar[rr]|{i_{*}}&& \mod\A\ar@/_-12pt/[ll]|{i^{!}}\ar@/_12pt/[ll]|{i^{*}}\ar[rr]|{j^{*}}&&\X\ar@/_-12pt/[ll]|{j_*}\ar@/_12pt/[ll]|{j_!}
}
\]
satisfying the following assertions:
\begin{enumerate}[\rm ({AR.}1)]
\item $(j_{!},j^{*},j_{*})$ and $(i^{*},i_{*},i^{!})$ are adjoint triples; 
\item the functors $i_{*},\, j_{!}$ and $j_{*}$ are fully faithful;
\item $\Im (i_{*})=\Ker(j^{*}).$
\end{enumerate}
Two Abelian recollements $\mathcal{R}=(\Y,\mod \A,\X)$ and $\mathcal{R}^{'}=(\Y',\mod \A,\X')$ of $\mod\A$ are said to be {\bf equivalent} if, denoting by $j^*\colon \mod\A\to \X$ and $(j^*)'\colon \mod\A\to \X'$ the functors appearing in the two recollements, respectively, there are equivalences $\Phi\colon \mod \A \to \mod\A$ and $\Psi\colon \X \to \X^{'}$ such that the following diagram commutes, up to a natural isomorphism:
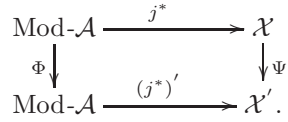
\[
\xymatrix@R=15pt@C=50pt{
\mod\A \ar[r]^{j^{*}} \ar[d]_{\Phi} & \X  \ar[d]^{\Psi}\\ 
\mod\A \ar[r]^{({j^{*}})^{'}} & \X^{'}.
}
\]
\end{definition}

Now, let us start with a TTF triple $(\C,\T,\F)$ in $\mod \A$ and let us hint how to construct the associated Abelian recollement. Indeed, one starts considering the Gabriel quotient  $q\colon \mod \A\to (\mod \A)/\T$ that, by \cite[Prop.\,I.1.3]{BR}, is equivalent to the full subcategory $\C\cap \F$ of $\mod \A$. Furthermore, the class $\T$ is both localizing and colocalizing, meaning that the quotient functor $\mod \A\to (\mod \A)/\T$ has both adjoints, thus the recollement induced by our TTF can be visualized by the following diagram
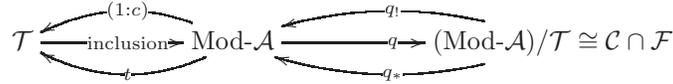
\[
\xymatrix{
\T\ar[rr]|{\text{inclusion}}&&\mod \A\ar@/_-12pt/[ll]|t\ar@/_12pt/[ll]|{(1:c)}\ar[rr]|q&&(\mod \A)/\T\cong \C\cap \F\ar@/_-12pt/[ll]|{q_*}\ar@/_12pt/[ll]|{q_!}
}
\]
%The {recollements} of Abelian categories are particularly nice decompositions of a given Abelian category by two other Abelian categories: %In this subsection, we recall the relation between the TTF triples in $\A$ and recollements of $\A$ (see Proposition \ref{prop.bijection-recollements-TTF}). For this purpose, we start with the following definition. 
%
%
%
%
%, see \cite[Thm.4.3]{PV}. 
The following result is a direct consequence of  \cite[Thm.\,4.3 and Coro.\,4.4]{PV}:

\begin{proposition} \label{prop.bijection-recollements-TTF}
Let $\mathcal{A}$ be a small preadditive category. For each Abelian recollement 
\[
\mathcal{R}:=\xymatrix{
\Y\ar[rr]|{i_{*}}&& \mod \A\ar@/_-12pt/[ll]|{i^{!}}\ar@/_12pt/[ll]|{i^{*}}\ar[rr]|{j^{*}}&&\X\ar@/_-12pt/[ll]|{j_*}\ar@/_12pt/[ll]|{j_!}
}
\]
there is an associated TTF triple $\tau_\mathcal{R}=(\Ker(i^*),\Im (i_*),\Ker(i^!))$ in $\mod \mathcal{A}$. The assignement $\mathcal{R}\mapsto \tau_\mathcal{R}$ defines a bijection from the set of equivalence classes of Abelian recollements of $\mod \mathcal{A}$ to the set of TTF triples in $\mod\mathcal{A}$. %Moreover, this map is bijective whenever $\mathcal{A}$ has enough projectives and enough injectives. 
\end{proposition}

%
%{\bf Aqu\'i debemos definir lo que es un recollement de categor\'ias Abelianas, definir lo que son recollements equivalentes de una misma categorÃ­a Abeliana}
%
%The following result is a direct consequence of  \cite[Theorem 4.3 and Corollary 4.4]{PV}
%
%\begin{proposition} \label{prop.bijection-recollements-TTF}
%Let $\mathcal{A}$ be an (Ab.$3$) Abelian category. To each recollement $\mathcal{Y}\equiv\mathcal{A}\equiv\mathcal{X}$ ($\mathcal{R}$) ({\bf escribir apropiadamente con los funtores $j_!, j^*,j_*,i^*,i_*,i^!$}) there is associated a TTF triple $\tau_\mathcal{R}=(\Ker(i^*),\Im(i_*),\Ker(i^!))$ in $\mathcal{A}$. The assignement $\mathcal{R}\rightarrow\tau_\mathcal{R}$ defines an injective map from the set of equivalence classes of Abelian recollements of $\mathcal{A}$ to the set of TTF triples in $\mathcal{A}$. Moreover, this map is bijective whenever $\mathcal{A}$ has enough projectives and enough injectives. 
%\end{proposition}

\subsection{Jan's Theorem}\label{jans_subs}

%Let $\A$ be a small preadditive category. A subclass $\T\subseteq \mod \A$ is said to be a {\bf torsion-torsionfree class} (or a {\bf TTF class}) if it is both a torsion and a torsionfree class. Equivalently, one can say that $\T$ is a TTF class if it is closed under taking submodules, quotients, products and extensions. Any TTF class $\T$ has an associated {\bf TTF triple} $(\C,\T,\F)$, where $(\C,\T)$ is a torsion pair, and $(\T,\F)$ is a hereditary torsion pair; in particular, $\C={}^{\perp}\T$, $\F=\T^{\perp}$ and $\T=\C^{\perp}={}^{\perp}\F$.
%
%\medskip
Analogously to what happens in categories of modules over associative unital rings, TTF triples in $\mod \A$ are in bijection with idempotent ideals of $\A$ (see Thm.\,\ref{prop.direct proof} below). In the following lemma we show how any TTF triple determines an idempotent ideal of $\A$, while the opposite direction is illustrated in Lem.\,\ref{idempont_implies_TTF}.

\begin{lemma}\label{prop.torsion ideal is two-sided}
Let  $(\C,\T,\F)$ be a TTF triple in $\mod{\A}$, denote by $c\colon \mod{\A}\to\mod{\A}$ the radical associated with the torsion pair $(\C,\T)$, and let $c(\A)$ be the ideal defined in Lem.\,\ref{torsion_ideal_lem}. Then, the following statements hold true:
\begin{enumerate}[\rm (1)]
\item $c(\A)(a,b)=\mathcal{I}_\T(a,b):=\{(r\colon a\to b)\in \A: T(r)=0,\, \forall T\in\T\}$;
\item the ideal $\mathcal{I}_\T$ defined in part (1) is idempotent. 
\end{enumerate}
\end{lemma}
\begin{proof}
%(1) Let $r\in c(\A)(a,b)$, $s_1\in \A(a',a)$, $s_2\in \A(b,b')$,  and let us verify that $r\circ s_1\in c(\A)(a',b)$ and that $s_2\circ r\in c(\A)(a,b')$. Indeed, since the torsion radical  $c$ is an additive subfunctor of $\id_{\mod{\A}}$, then $\A(-,s_2)(c(\A(-,b)))\leq c(\A(-,b'))$. In particular, $r\circ s_2=\A(a,s_2)(r)\in  c(\A(-,b'))(a)=c(\A)(a,b')$. On the other hand, the fact that $r\in c(\A)(a,b)$, means exctly that the image of the map $\A(-,r)\colon \A(-,a)\to \A(-,b)$ is torsion, thus we can corestric to obtain a map $\A(-,r)\colon \A(-,a)\to c(\A(-,b))$. It is now clear that also the composition $\A(-,r)\circ\A(-,s_1)\colon\A(-,a')\to \A(-,b)$ also takes values in $c(\A(-,b))$, so that $r\circ s_1\in c(\A(a',b))$.
It is clear that $\I_\T$ is a two-sided ideal. For a module $M\in\mod \A$, let
\[
\Rej_\T(M):=\bigcap\{\Ker(\mu):{T\in\T,\,\mu\in\hom_\A(M,T)}\};
\]
we claim that $\mathcal{I}_\T(-,a)=\Rej_\T(H_a)$, for all $a\in\Ob(\A)$. Indeed, by the Yoneda Lemma, $\Hom_\A(H_a,T)\cong T(a)$ and, identifying each $\mu\colon H_a\to T$ with the corresponding element of $T(a)$, we readily see that $\Rej_\T(H_a)(b)=\{\alpha\in\A(b,a):T(\alpha)(\mu)=0,\, \forall T\in\T,\, \mu\in T(a)\}$. That is,  
\[
\Rej_\T(H_a)(b)=\{\alpha\in\A(b,a):T(\alpha)=0,\, \forall T\in\T\}=\mathcal{I}_\T(b,a). 
\]
Being $\T$ closed under products and submodules, $M/\Rej_\T(M)\in \T$, for all $M\in\mod{\A}$. In particular,  $H_a/\I_\T(-,a)\in\T$, for all $a\in\Ob(\A)$. Let now $M\in\mod \A$, consider an epimorphism $\coprod_{i\in I}H_{a_i}\to M$, and take the composition with the natural projection $M\to M/M \I_\T$: 
\[
\varphi\colon \coprod_{i\in I}H_{a_i}\to M\to M/M \I_\T.
\]
Denote by $\pi\colon \A\to \A/\I_\T$ the projection, then $M/M\I_\T\cong \pi_*\pi^*(M)$ and 
\begin{align*}
\hom_\A(H_{a_i},\pi_*\pi^*(M))&\cong \hom_{\A/\I_\T}(\pi^*H_{a_i},\pi^*(M))\\
&\cong \hom_{\A/\I_\T}(H_{a_i}/\I_\T(-,a_i),M/M\I_\T),
\end{align*}
see \eqref{extending_representables_eq} and the discussion there. Thus, each component $\varphi_i:H_{a_i}\to M/M\I_\T$ factors through $H_{a_i}/\I_\T(-,a_i)$, so $M/M \I_\T$ is a quotient of a module of the form $\coprod_{i\in I}(H_{a_i}/\I_\T(-,a_i))\in \T$, hence, $M/M \I_\T\in \T$. In particular, $\mathcal{I}_\T(-,a)/\mathcal{I}^2(-,a)\in\T$ for all $a\in \Ob(\A)$. Considering now the following short exact sequence  in $\mod{\A}$:
\[
\xymatrix@C=15pt{
0\ar[r]&{\mathcal{I}_\T(-,a)}/{\mathcal{I}_\T^2(-,a)}\ar[rr]&&H_a/{\mathcal{I}_\T^2(-,a)}\ar[rr]&&H_a/{\mathcal{I}_\T(-,a)}\ar[r]& 0
}
\]
 whose outer terms are in $\T$. We then get that $H_a/\mathcal{I}_\T^2(-,a)\in \T$. Hence, $\Rej_\T(H_a)=\mathcal{I}_\T(-,a)\subseteq \Ker(H_a\to H_a/\mathcal{I}_\T^2(-,a))=\mathcal{I}_\T^2(-,a)$, proving that $\mathcal{I}_\T$ is idempotent. 
\end{proof}

\begin{lemma}\label{idempont_implies_TTF}
Given an idempotent ideal $\mathcal{I}$ of $\A$, the following is a TTF class in $\mod{\A}$:
\[
\T_\I:=\{T\in \mod \A:T(\alpha)=0,\,\forall \alpha\in \I(a,b),\, a,b\in\A\}.
\]
\end{lemma}
\begin{proof}
Since products and coproducts in $\mod{\A}$ are computed ``pointwise'', it is clear that $\T_\I$ is closed under taking products, coproducts, submodules and quotients. It only remains to check that it is closed under extensions. Indeed, consider an exact sequence $0\to T'{\to}T{\to}T''\to 0$ in $\mod{\A}$, where $T',T''\in\T_\I$, and take a morphism $\alpha\in\mathcal{I}(a,b)$. Being $\mathcal{I}$ idempotent, we can write $\alpha =\sum_{j=1}^n\gamma_j\circ\beta_j$, with $\beta_j\in\mathcal{I}(a,c_j)$ and $\gamma_j\in\mathcal{I}(c_j,b)$, with $n\in\N$, $j\in\{1,\dots,n\}$ and $c_j\in \Ob(\A)$. Our goal is to prove that $T(\alpha)=0$, for which it is enough to prove that $T(\gamma_j)\circ T(\beta_j)=0$ for all $j=1,\dots,n$. Therefore, it is not restrictive to assume that $\alpha=\gamma\circ\beta$, for some morphisms $\beta\in\mathcal{I}(a,c)$ and $\gamma\in\mathcal{I}(c,b)$. Due to the definition of $\T_\I$, we have the following commutative diagram with exact rows in $\Ab$:
\[
\xymatrix{ 
0 \ar[r] & T^{'}(a) \ar[r]^{u_a} \ar[d]^{0} & T(a) \ar[r]^{p_a} \ar[d]^{T(\beta)} & T^{''}(a) \ar[r] \ar[d]^{0} & 0\\ 
0 \ar[r] & T^{'}(c) \ar[r]_{u_c} & T(c) \ar[r]_{p_c} & T^{''}(c) \ar[r] & 0
}
\]
By  the universal property of co/kernels, we get a morphism  $\varphi \colon T''(a)\to T'(c)$ in $\Ab$ such that $T(\beta)=u_c\circ\varphi\circ p_a$. By a similar argument, we get a morphism $\psi \colon T''(c)\to T'(b)$ such that $T(\gamma)=u_b\circ\psi\circ p_c$. Hence, $T(\alpha )=T(\gamma)\circ T(\alpha)=u_b\circ\psi\circ p_c\circ u_c\circ\varphi\circ p_a$, which is the zero morphism since $p_c\circ u_c=0$. 
\end{proof}

\begin{theorem} \label{prop.direct proof}
Let $\A$ be a small preadditive category. Then there is a one-to-one correspondence
\[
\xymatrix@R=0pt{
\Phi:\S_1:=\{\text{Idempotent ideals of $\A$}\}\ar@{<->}[rr]^(.55){1:1}&&\{\text{TTFs in $\mod{\A}$}\}=:\S_2:\Psi\\
\I\ar@{|->}[rr]&&\Phi(\I):=\T_\I\\
\Psi(\T):=\mathcal{I}_\T\ar@{<-|}[rr]&&\T
}
\]
where $\I_\T$ and $\T_\I$ are defined as in Lem.\,\ref{prop.torsion ideal is two-sided} and \ref{idempont_implies_TTF}.
\end{theorem}
\begin{proof}
Given a TTF class $\T$ in $\mod \A$ it is not hard to check that $\T\subseteq (\Phi\circ\Psi)(\T)$,  so let us verify the converse inclusion. Indeed, if $M\in (\Phi\circ\Psi)(\T)$, then we have $M(\alpha)=0$, for any morphism $\alpha\in\mathcal{I}_\T(a,b)$, $a,\,b\in\Ob(\A)$; equivalently, any morphism $\A(-,b)\to M$ vanishes on $\mathcal{I}_\T(-,b)$. This means that $M$ can be written as a quotient of a module of the form $\coprod_{i\in I}(H_{a_i}/\I_\T(-,a_i))$ but such a module belongs to $\T$ (see the proof of Lem.\,\ref{prop.torsion ideal is two-sided}), so that $M\in \T$.

On the other hand, given an idempotent ideal $\I$ of $\A$, one checks easily that $\mathcal{I}\subseteq (\Psi\circ\Phi)(\mathcal{I})$. Now, let $\alpha\colon a\to b$ be a morphism in $\alpha\in (\Psi\circ\Phi )(\mathcal{I})(a,b)$, then $T(\alpha)=0$ for all $T\in\T_\I=\Psi (\mathcal{I})$. In particular, for $T=H_b/\mathcal{I}(-,b)$, the equality $T(\alpha)=0$ means that the induced map 
\[
{\A(b,b)}/{\mathcal{I}(b,b)}\to{\A(a,b)}/{\mathcal{I}(a,b)}\qquad\text{such that}\qquad\bar{\beta}\mapsto\overline{\beta\circ\alpha},
\] 
is trivial. Therefore, $\bar{\alpha}=\overline{\id_b\circ\alpha}=\bar{0}$ or, equivalently, $\alpha\in\mathcal{I}(a,b)$. This proves that $(\Psi\circ\Phi)(\mathcal{I})(a,b)\subseteq\mathcal{I}(a,b)$, so that $(\Psi\circ\Phi) (\mathcal{I})=\mathcal{I}$. 
\end{proof}

As a byproduct of the above proofs we obtain the following: given a TTF triple $(\C,\T,\F)$, there is a uniquely associated ideal $\I_\T$ such that 
\[
\T=\Gen\{H_a/\I_\T(-,a):a\in \Ob(\A)\}\cong \mod {\A/\I_\T}\quad\text{ and }
\]
\[
\C=\Gen\{\I_\T(-,a):a\in \Ob(\A)\}.
\]

\subsection{Abelian recollements of module categories}\label{subs_Abelian recollements of module categories}

Let us start with the following definition:

\begin{definition} \label{def.TTF triple generated}
Let $\tau=(\mathcal{C},\mathcal{T},\mathcal{F})$ be a TTF triple in $\mod \A$ and let $\mathcal{S}$ be a class of right $\mathcal{A}$-modules. We say that $\tau$ is {\bf generated by $\mathcal{S}$} when the torsion pair $(\mathcal{C},\mathcal{T})$ is generated by $\mathcal{S}$, i.e., when $\mathcal{T}=\mathcal{S}^\perp$. Furthermore, we say that $\tau$ is {\bf generated by finitely generated projective $\mathcal{A}$-modules} when it is generated by a set of finitely generated projective objects of $\mod \A$. 
\end{definition}

The key result of this section is the following proposition:

\begin{proposition} \label{prop.TTF triples generated by fg projectives}
Let $\mathcal{A}$ be a small preadditive category, let $\tau =(\mathcal{C},\mathcal{T},\mathcal{F})$ be a TTF triple in $\mod \A$ and let $\mathcal{I}$ be the associated idempotent ideal of $\mathcal{A}$. The following assertions are equivalent:
\begin{enumerate}[\rm (1)]
\item $\mathcal{C}\cap\mathcal{F}$ is equivalent to the module category over a small preadditive category;
\item $\tau$ is generated by finitely generated projective $\mathcal{A}$-modules;
\item $\mathcal{I}$ is the trace in $\mathcal{A}$ of a set of finitely generated projective right $\mathcal{A}$-modules (see Lem.\,\ref{prop_char_ideals_trace_fgp} for other equivalent characterizations of these ideals). 
\end{enumerate}
\end{proposition}
\begin{proof}
Let us fix the following notation for the recollement induced by $\tau$:
\[
\mathcal{R}:\qquad
\xymatrix{
\mod{({\A}/{\I})}\cong\Y\ar[rr]|{i_{*}}&&\mod \A\ar@/_-12pt/[ll]|{i^{!}}\ar@/_12pt/[ll]|{i^{*}}\ar[rr]|{j^{*}}&&\X\cong \C\cap \F\ar@/_-12pt/[ll]|{j_*}\ar@/_12pt/[ll]|{j_!}
}
\]
%where $\mathcal{Y}=\mod(\mathcal{A}/\mathcal{I})\cong\mathcal{T}$, $\mathcal{X}$ is the quotient category $(\mod\mathcal{A})/\mathcal{T}$, the functor $i_*:\mod(\mathcal{A}/\mathcal{I})\longrightarrow\mod\mathcal{A}$ is the restriction of scalars with respect to the projection functor $i:\mathcal{A}\longrightarrow\mathcal{A}/\mathcal{I}$ and $j^*:\mod\mathcal{A}\longrightarrow (\mod\mathcal{A})/\mathcal{T}$ is the Gabriel quotient functor. In particular, we have that $\Ker(j^*)=\mathcal{T}$.  Note that $i$ was denoted by $\pi_I$ in Subsection \ref{subs_tors_and_id}, but, for the contents of this section, we prefer to adhere to the usual terminology for recollements. 
%By \cite[Proposition I.1.3]{BR},  we have an equivalence of categories  $(\mod\mathcal{A})/\mathcal{T}\cong\mathcal{C}\cap\mathcal{F}$. The composition $\mod\mathcal{A}\stackrel{j^*}{\longrightarrow}(\mod\mathcal{A})/\mathcal{T}\stackrel{\cong}{\longrightarrow}\mathcal{C}\cap\mathcal{F}$ takes $M\rightsquigarrow c((1:t)(M))$, where $c$ and $(1:t)$ are the radical and coradical, respectively, associated to the torsion pairs $(\mathcal{C},\mathcal{T})$ and $(\mathcal{T},\mathcal{F})$. We replace  $\mathcal{X}$ by $\mathcal{C}\cap\mathcal{F}$ in the sequel, so that $j^*:\mod\mathcal{A}\longrightarrow\mathcal{C}\cap\mathcal{F}$ takes $M\rightarrow c((1:t)(M))$. 
%
%We now pass to prove the implications of the proposition.

\noindent
(1)$\Rightarrow$(2). By Rem.\,\ref{rem_morita_cauchy}, we can
%Recall that an Abelian category is equivalent to $\mod\mathcal{B}$, for some small preadditive category $\mathcal{B}$ if,  and only if, it is (Ab.$3$) and has a set of compact (=small) projective generators. Note that these last objects are also finitely presented, i.e. for each of them the corresponding covariant $\Hom$ functor preserves direct limits.  Let us 
fix a set $\mathcal{P}$ of finitely generated projective generators of $\mathcal{C}\cap\mathcal{F}$. Being $j_!\colon\mathcal{C}\cap\mathcal{F}\to\mod\mathcal{A}$ the left adjoint of a functor which preserves all limits and colimits, it preserves finitely generated projective objects, so that $j_!(\P)\subseteq \proj(\A)$; let us show that $\T=j_!(\P)^{\perp}$. Indeed, $T\in j_!(\mathcal{P})^\perp$ if and only if 
$0=\hom_\A(j_!P,T)\cong 
(\mathcal{C}\cap\mathcal{F})(P,j^*T)$, for all $P\in\P$. But this is equivalent to say that 
$j^*T=0$ since $\mathcal{P}$ generates $\mathcal{C}\cap\mathcal{F}$. 
Therefore, $T\in j_!(\mathcal{P})^\perp$ if and only if 
$T\in\Ker(j^*)=\mathcal{T}$.
%
%
%
%
%
%
%using that $\Im (j_{!})\subseteq \C$, one sees that $j_!(\mathcal{P})\subseteq \C$, so $\T=\C^{\perp}\subseteq j_!(\P)^{\perp}$. On the other hand, given $N\in j_!(\P)^{\perp}$,  one obtains by adjunction that $(\C\cap \F)(P,j^*(N))=0$ for all $P\in\P$, so that $j^*(N)=0$ being $\P$ a set of generators. Hence,  $N\in\ker(j^*)=\mathcal{T}$. %It follows that the TTF triple is generated by $j_!(\mathcal{P})$. 

\smallskip\noindent
(2)$\Rightarrow$(1). Let $Y\in \mathcal{C}\cap\mathcal{F}$, and fix a set $\mathcal{Q}$ of finitely generated projective $\mathcal{A}$-modules that generates $\tau$. We then have that $\mathcal{C}=\Gen(\mathcal{Q})$, so there is an epimorphism $p\colon\coprod_{\Lambda}Q_\lambda\to j_!(Y)$, for some family $(Q_\lambda)_{\Lambda}\subseteq\mathcal{Q}$. We obtain the following epimorphism in $\mathcal{C}\cap\mathcal{F}$:
\[
\xymatrix{
\coprod_{\Lambda}j^*(Q_\lambda)\cong j^*(\coprod_{\Lambda}Q_\lambda)\stackrel{j^*(p)}{\longrightarrow} (j^*\circ j_!)(Y)\cong Y.
}
\] 
We have then reduced our task to prove that  $j^*(\mathcal{Q})$ consists of small projective objects of $\mathcal{C}\cap\mathcal{F}$. 
For this, note that, although kernels and cokernels in $\mathcal{C}\cap\mathcal{F}$ are not computed as in $\mod\mathcal{A}$,  epimorphisms and monomorphisms in  $\mathcal{C}\cap\mathcal{F}$ are precisely the morphisms which are epimorphisms and monomorphisms, respectively, in $\mod\mathcal{A}$. Hence, $j^*(Q)\cong (1:t)(Q)$ is projective in $\mathcal{C}\cap\mathcal{F}$, for each $Q\in\Q$. Moreover, since coproducts in $\mathcal{C}\cap\mathcal{F}$ are computed as in $\mod\mathcal{A}$, it is clear that $(1:t)(Q)$ is small in $\mathcal{C}\cap\mathcal{F}$. Since $\mathcal{C}\cap\mathcal{F}$ is an (Ab.$5$) (see \cite[Prop.\,3.5]{PaV}), so in particular (Ab.$3$), Abelian category, we conclude that it is equivalent to a category of modules. 

\smallskip\noindent
(2)$\Rightarrow$(3). Let $\mathcal{P}$ be a set in $\mathcal{C}\cap\proj(\mathcal{A})$ which generates $\tau$, and let $\mathcal{I}':=\tr_{\P}(\A)$. It is clear that $c(M)=\tr_\mathcal{P}(M)$, for all $M\in\mod\mathcal{A}$. Then, by definition of $\mathcal{I}'$, we have that $\mathcal{I}'(-,b)=\tr_\mathcal{P}(H_b)=c(H_b)$. By Lem.\,\ref{prop.torsion ideal is two-sided}, $c(H_b)=\mathcal{I}(-,b)$. It follows that $\mathcal{I}'=\mathcal{I}$. 

\smallskip\noindent
(3)$\Rightarrow$(2). Let $\mathcal{P}$ be a set of finitely generated projective $\mathcal{A}$-modules such that $\mathcal{I}=\tr_\P(\A)$. We know that $\mathcal{C}=\Gen\{\mathcal{I}(-,a):a\in\text{Ob}(\mathcal{A})\}$ (see the comment after Thm.\,\ref{prop.direct proof}) and each $\mathcal{I}(-,a)=\tr_\mathcal{P}(H_a)$ is epimorphic image of a coproduct of objects of $\mathcal{P}$. We then get that $\mathcal{C}\subseteq\Gen(\mathcal{P})$, so that $\Gen(\mathcal{P})^{\perp}=\mathcal{P}^\perp \subseteq \mathcal{C}^\perp=\T$. 
\\
For the converse inclusion, let $b\in \Ob(\A)$ and $P\in \P$. By the projectivity of $P$, any morphism $\phi\colon P\to H_{b}/\tr_{\P}(H_b)$ lifts to a morphism $P\to H_b$, but any such morphism factors through $\tr_\P(H_b)$, so $\phi=0$. We have just verified that $H_{b}/\tr_{\P}(H_b)=H_{b}/\mathcal{I}(-,b)\in \P^{\perp}$ for all $b\in\Ob(\A)$. Now, given $T$ in $\T$, we know (see once again the comment after Thm.\,\ref{prop.direct proof}) that there is an epimorphism $p\colon \coprod_{\Lambda}H_{b_\lambda}/\mathcal{I}(-,b_\lambda)\to T$. %By the above discussion, we deduce that $\mathcal{I}(-,b_\lambda)=\tr_{\mathcal{P}}(H_{b_\lambda})$, for all $\lambda \in \Lambda$. 
Using the projectivity of the objects in $\mathcal{P}$, we get that $T\in \mathcal{P}^{\perp}$, and so $\T\subseteq \mathcal{P}^{\perp}$. Therefore, $\mathcal{T}=\mathcal{P}^\perp$, and hence $\tau$ is generated by $\mathcal{P}$.
\end{proof}

We immediately derive the main result of the section.

\begin{theorem}\label{main_thm_recollement}
Let $\mathcal{A}$ be a small preadditive category. There are one-to-one correspondences between:
\begin{enumerate}[\rm (1)]
\item the equivalence classes of recollements of $\mod \A$ by module categories over small preadditive categories;
\item the TTF triples in $\mod \A$ generated by finitely generated projective $\mathcal{A}$-modules;
\item the idempotent ideals of $\mathcal{A}$ which are the trace of a set of finitely generated projective $\mathcal{A}$-modules;
\item the idempotent ideals of the  additive closure $\widehat{\mathcal{A}}$ of $\mathcal{A}$ generated by a set of idempotent endomorphisms;
\item the full subcategories of the Cauchy completion $\widehat{\mathcal{A}}_\oplus$ which are closed under coproducts and summands.
\end{enumerate}
\end{theorem}
\begin{proof}
The bijection between the families described in parts (1), (2) and (3) is given by Prop.\,\ref{prop.TTF triples generated by fg projectives}.  The bijection between the families in (3) and (4) is given by Prop.\,\ref{prop_char_ideals_trace_fgp}. Furthermore, by Coro.\,\ref{ADD_morita} we know that $\mod \A$, $\mod{\widehat\A}$ and $\mod {\widehat \A_{\oplus}}$ are equivalent categories, which implies that the sets of TTF triples in these categories coincide, and they are in bijection with the set of idempotent ideals of $\widehat{\mathcal{A}}$ (resp., $\widehat\A_{\oplus}$) that are the trace of a set of finitely generated projective modules. By Prop.\,\ref{prop_char_ideals_trace_fgp}, this family is in bijection with the full subcategories of the Cauchy completion $\widehat{\mathcal{A}}_\oplus$ which are closed under coproducts and summands. 
\end{proof}

In the above theorem we have completely characterized those ideals $\I$ that induce a recollement of $\mod \A$ by categories of modules. In the rest of this subsection we are going to describe a standard form for such recollements. Let us start with the following definition:

\begin{definition}
Let $\A$ be a preadditive category and $\E$ a set of idempotent endomorphisms in $\mathcal{A}$. The {\bf corner category of $\mathcal{E}$ in $\mathcal{A}$}, denoted by $\mathcal{C}_\mathcal{E}$ in the sequel, is defined as follows:
\begin{itemize}
\item $\Ob(\mathcal{C}_\mathcal{E}):=\mathcal{E}$;
\item given two morphisms $\epsilon\colon x\to  x$ and $\epsilon'\colon x'\to  x'$ in $\mathcal{E}$, where $x,\, x'\in\Ob(\mathcal{A})$, then $\mathcal{C}_\mathcal{E}(\epsilon,\epsilon')$ is the subgroup of $\mathcal{A}(x,x')$ of those morphisms $\alpha \colon x\to  x'$ in $\mathcal{A}$ that admit a decomposition $\alpha =\epsilon'\circ\beta\circ\epsilon$, for some morphism $\beta\in\mathcal{A}(x,x')$;
\item composition of morphisms is $\mathcal{C}_\mathcal{E}$ is defined as in $\mathcal{A}$.
\end{itemize} 
\end{definition}

We leave to the reader the verification that $\mathcal{C}_\mathcal{E}$ is a well-defined preadditive category. If the ambient category is Cauchy complete, corner categories have a particularly simple description:

\begin{lemma}
Let $\A$ be a Cauchy complete preadditive category and let $\E$ be a set of idempotent endomorphisms in $\mathcal{A}$. For each $\epsilon\colon x\to x$ in $\E$, consider the splitting $x\to x_\epsilon\to x$ of $x$ induced by $\epsilon$. Then, $\C_\E$ is equivalent to the full subcategory of $\A$ having as objects the $x_\epsilon$, with $\epsilon\in\mathcal{E}$. 
\end{lemma}
\begin{proof}
Let us define a functor $F\colon \X_\E\to \C_\E$, defined on objects by the rule $x_\epsilon\mapsto \epsilon$ and that maps a morphism $\alpha\colon x_{\epsilon}\to x_{\epsilon'}$ in $\X_\E$ to $F(\alpha):=\epsilon'\circ\alpha\circ \epsilon$. This functor is clearly essentially surjective and it is an exercise to verify that it is also fully faithful.
\end{proof}

We have now the following  version of \cite[Thm.\,5.3]{PV}  for preadditive categories:

\begin{corollary} \label{cor.PV generalization}
Let $\mathcal{A}$ be a small preadditive category. Let $\tau=(\C,\T,\F)$ be a TTF associated to a recollement of $\mod \A$ by categories of modules over small preadditive categories. Then, this recollement is equivalent to one of the form
\[
\mathcal{R}:\qquad
\xymatrix{
\mod{({\B}/{\B\E\B})}\ar[rr]|{i_{*}}&&\mod \B\ar@/_-12pt/[ll]|{i^{!}}\ar@/_12pt/[ll]|{i^{*}}\ar[rr]|{j^{*}}&&\mod{\C_\E}\ar@/_-12pt/[ll]|{j_*}\ar@/_12pt/[ll]|{j_!}
}
\]
where $\mathcal{B}$ is a preadditive category Morita equivalent to $\mathcal{A}$ and $\mathcal{E}$ is a set of idempotent endomorphisms in $\mathcal{B}$. 
\end{corollary}
\begin{proof}
Let $\B:=\widehat\A_{\oplus}$ be the Cauchy completion of $\A$. Then, by Prop.\,\ref{prop.TTF triples generated by fg projectives}, we know that there is a full subcategory $\X$ of $\B$, closed under summands and coproducts, such that our recollement is induced by the idempotent ideal $\I:=\B\X\B$. Notice that $\C_\X\cong \X$ is a (Cauchy complete) preadditive category; to conclude the proof it is enough to show that $\C\cap \F\cong \mod \X$. For this, it is enough to show that the full subcategory of finitely generated projective objects in $\C\cap \F$ is generating and equivalent to $\X$. We know that $\X$ is equivalent to the full subcategory of $\mod \A$ spanned by $\{H_x:x\in \X\}$. Now, one can prove exactly as in the implication ``(2)$\Rightarrow$(1)'' of Prop.\,\ref{prop.TTF triples generated by fg projectives}, that the full subcategory $\{c(H_x):x\in \X\}$ of $\C\cap \F$ is again equivalent to $\X$, it generates $\C\cap \F$, and it consists of finitely generated projectives. 
\end{proof}

It is natural to ask what happens when the side categories in the recollement of the above corollary are actually module categories over unital rings:

\begin{corollary} \label{cor.sides are unital modcats}
Let $\mathcal{A}$ be a small preadditive category whose module category admits a recollement 
\[
\mathcal{R}:\qquad
\xymatrix{
\mod{B}\ar[rr]|{i_{*}}&&\mod \A\ar@/_-12pt/[ll]|{i^{!}}\ar@/_12pt/[ll]|{i^{*}}\ar[rr]|{j^{*}}&&\mod{C}\ar@/_-12pt/[ll]|{j_*}\ar@/_12pt/[ll]|{j_!}
}
\]
where $B$ and $C$ are  associative and unital rings. Then there is an associative unital ring $A$ such that $\mod \A$ is equivalent to $\mod A$. Hence, any such recollement is equivalent to a recollement as the one described in \cite[Thm.\,5.3]{PV}.
\end{corollary}
\begin{proof}
Up to equivalence, we can suppose that $\A$ is Cauchy complete and we can let $\B\cong \proj(B)$ and $\C\cong \proj(C)$ be the Cauchy completions of $B$ and $C$, respectively. Now, by Rem.\,\ref{rem_morita_cauchy}, we have $\oplus$-generators $b$ and $c$ in $\B$ and $\C$, respectively. In fact, we can identify $\C$ with a full subcategory (closed under coproducts and summands) of $\A$ and, up to this identification, $\B\cong \A/\A\C\A$ (where $\A\C\A$ is the ideal of $\A$ generated by the identities of objects in $\C$). Identifying $\B$ with $\A/\A\C\A$ we can consider both $b$ and $c$ as objects in $\A$. To conclude, it is enough to verify that $a:=b\oplus c$ is an $\oplus$-generator in $\A$, so that $\mod \A\cong \mod {\A(a,a)}$. For this, let $x\in \Ob(\A)=\Ob(\A/\A\C\A)$, then $x$, when viewed as an object in $\A/\A\C\A$, is a summand of a finite coproduct of copies of $b$, that is, there is $n\in\N$ such that
\[
\xymatrix@C=10pt{
x\ar[rr]^(.4){\id_x}&&x& =& x\ar[rr]^{}&&b^n\ar[rr]^{}&&x,&&\text{in $\A/\A\C\A$.}
}
\]
This means exactly that there exists $c'\in\Ob(\C)$ such that
\[
\xymatrix@R=2pt@C=10pt{
&&&&&&b^n\ar@{}[dd]|+\ar[drr]^{}\\
x\ar[rr]^(.4){\id_x}&&x& =& x\ar[drr]\ar[rru]^{}&&&&x,&&\text{in $\A$.}\\
&&&&&&c'\ar[urr]
}
\]
Now, since $c$ is a $\oplus$-generator in $\C$, $c'$ is a summand of $c^m$ for some $m\in \N$. Hence, the identity of $x$ factors through $(b\oplus c)^{\max\{m,n\}}$, so that $b\oplus c$ is a $\oplus$-generator for $\A$.
\end{proof}

\subsection{Centrally splitting TTF's}\label{subs_centrally}\label{coro_central_id}

In the following proposition we show that the TTF triples that arise from a central idempotent  are exactly the split ones. This proves Corollary D in the Introduction. For a more general version of the following result, in the setting of idempotent complete additive categories, we refer to \cite[Prop.\,1.7.4]{N}. We include here a complete proof, in our particular setting, since it easily follows as a consequence of the results of the previous subsections.

\begin{proposition}
Consider a TTF triple $(\C,\T,\F)$ in $\mod \A$, denote by $\I$ the associated idempotent ideal, and denote by $c\colon \mod \A\to \C$ and $t\colon \mod \A\to \T$ the associated torsion radicals.
 The following are equivalent 
\begin{enumerate}[\rm (1)]
\item for any module $M$, there is a decomposition $M=c(M)\oplus t(M)$;
\item there is an idempotent element $\epsilon\in Z(\A)$ such that $\I=\I_\epsilon$;
\item $\C=\F$.
\end{enumerate}
In this case we have
 \[
\C=\{M: M(\epsilon_a)\text{ is an iso, for all }a\in\Ob(\A)\}\quad\text{and}
\]
\[\T=\{M: M(\epsilon_a)=0\text{, for all }a\in\Ob(\A)\}.
\]
\end{proposition}
\begin{proof}
%\textcolor{red}{\bf{Creo que la implicaci\'on $1) \Longleftrightarrow 3)$ es v\'alida para cualquier categor\'ia Abeliana. Les escribo a continuaci\'on el argumento: \newline
%No es d\'ificil darse cuenta que en cualquier categor\'ia Abeliana, una clase de torsi\'on es cerrada bajo cocientes y una clase libre de torsi\'on es cerrada bajo subobjetos. Para todo $M$ objeto de la categor\'ia en cuesti\'on, consideremos el siguiente diagrama:
%$$\xymatrix{0  \ar[r] & K \ar[r] \ar@{^(->}[d] \pushoutcorner & c(M) \ar[r] \ar@{^(->}[d] & Z \ar[r] \ar^{\alpha}[d] & 0 \\ 0 \ar[r] & t(M) \ar[r] & M \ar@{>>}[d] \ar[r] & (1:t)(M) \ar[r] \ar@{>>}[d] & 0 \\ && (1:c)(M) \ar@{>>}[r] & W }$$
%(3)$\Rightarrow$(1) Si $\C=\F$, entonces $K\in \T \cap \F$ y $W=\T \cap \C$, por lo que $K=0=W$ y as\'i $Z \cong (1:t)(M) \cong c(M)$. Tambi\'en se deduce que la sucesi\'on del medio escinde y por tanto, $M=t(M)\oplus (1:t)(M)=t(M)\oplus c(M).$ La implicaci\'on (1)$\Rightarrow$(3) la prueba que aparece en est\'a p\'agina sirve para cualquier categor\'ia Abeliana.}}
(3)$\Rightarrow$(1). Notice first that $M=c(M)+t(M)$, in fact, the inclusion $\F\subseteq \C$ is equivalent to say that 
$c(M/t(M))=M/t(M)$. Hence, 
\[
\frac{M}{t(M)}=c\left(\frac{M}{t(M)}\right)=\frac{M}{t(M)}\cdot\I=\frac{M\I+t(M)}{t(M)}=\frac{c(M)+t(M)}{t(M)}.
\]
 Furthermore, $\T$ is always hereditary, while $\C=\F$ is hereditary because $\F$ is closed under taking submodules, hence $c(M)\cap t(M)\in \T\cap \C=\T\cap \F=0$.

\noindent\smallskip
(1)$\Rightarrow$(3). Let $M\in\mod \A$; by the decomposition $M\cong t(M)\oplus c(M)$ we can see that $M\in \F$ if and only if $t(M)=0$ if and only if $M\cong c(M)$, if and only if $M\in\C$.

\noindent\smallskip
(1)$\Rightarrow$(2).  We have a decomposition $\A(-,a)=c(\A(-,a))\oplus t(\A(-,a))$, natural in $a$, for all $a\in\Ob(\A)$. It then follows a decomposition $\A=c(\A)\oplus t(\A)$ as ideals of $\A$. By Prop.\,\ref{prop.bijection-central-idempotents}, there exists a unique $\epsilon =\epsilon^2\in Z(\A)$ such that $\I=c(\A)=\I_\epsilon$.

\noindent\smallskip
(2)$\Rightarrow$(1). It is a consequence of Thm.\,\ref{prop.direct proof}, Prop.\,\ref{prop.bijection-central-idempotents} and Lem.\,\ref{lemma_direct_summands}.

\noindent\smallskip
For the last  statement apply Coro.\,\ref{cor.description-TTF} and Rem.\,\ref{rem.clarification}.
\end{proof}

\medskip
Manuel Saor\'in -- \texttt{msaorin@um.es}\\
{Departamento de Matem\'{a}ticas,
Universidad de Murcia,  Aptdo.\,4021,
30100 Espinardo, Murcia,
SPAIN}

\medskip
Carlos E. Parra -- \texttt{carlos.parra@uach.cl}\\
{Instituto de Ciencias Físicas y Matemáticas, Edificio Emilio Pugin, Campus Isla Teja, Universidad Austral de Chile, 5090000 Valdivia, CHILE}

\medskip
Simone Virili -- \texttt{s.virili@um.es} or \texttt{virili.simone@gmail.com}\\
{Departamento de Matem\'{a}ticas,
Universidad de Murcia,  Aptdo.\,4021,
30100 Espinardo, Murcia,
SPAIN}

\end{document}